\documentclass[11pt]{amsart}
\usepackage[top=0.9in, bottom=0.9in, left=1.0in, right=1.0in]{geometry}
\usepackage{amsthm}
\usepackage{enumitem}
\usepackage{moreenum}
\usepackage{amssymb,verbatim}
\usepackage{upgreek}
\usepackage{xspace,cmap}
\usepackage{tikz-cd}
\usepackage{csquotes}
\usepackage{stmaryrd} 
\usepackage[colorlinks,citecolor=blue,urlcolor=black,linkcolor=black]{hyperref}

\newtheorem*{rep@theorem}{\rep@title}
\newcommand{\newreptheorem}[2]{%
\newenvironment{rep#1}[1]{%
 \def\rep@title{#2 \ref{##1}}%
 \begin{rep@theorem}}%
 {\end{rep@theorem}}}

\renewcommand{\restriction}{\mathbin\upharpoonright}    

\newtheorem*{theorem*}{Theorem}
\newtheorem*{maintheorem*}{Main Theorem}
\newtheorem*{question*}{Question}
\newtheorem*{corollary*}{Corollary}
\newtheorem*{definition*}{Definition}

\newtheorem{theorem}{Theorem}[section]
\newreptheorem{theorem}{Theorem}

\newtheorem{claim}{Claim}[theorem]

\newtheorem{lemma}[theorem]{Lemma}
\newtheorem{corollary}[theorem]{Corollary}

\newtheorem{fact}[theorem]{Fact}

\newtheorem{manualtheoreminner}{\textbf{Theorem}}
\newenvironment{manualtheorem}[1]{%
  \IfBlankTF{#1}
    {\renewcommand{\themanualtheoreminner}{\unskip}}
    {\renewcommand\themanualtheoreminner{#1}}%
  \manualtheoreminner
}{\endmanualtheoreminner}

\theoremstyle{definition}

\newtheorem{definition}[theorem]{Definition}
\newtheorem{notation}[theorem]{Notation}

\theoremstyle{remark}
\newtheorem{remark}[theorem]{Remark}

\newcommand\HOD{\textnormal{HOD}}

\DeclareMathOperator{\Add}{Add}

\DeclareMathOperator{\supp}{supp}
\DeclareMathOperator{\crit}{crit}

\DeclareMathOperator{\Ult}{Ult}

    \newcommand{\one}{\mathop{1\hskip-3pt {\rm l}}}

\makeatletter
\newcommand{\tpitchfork}{%
  \vbox{
    \baselineskip\z@skip
    \lineskip-.52ex
    \lineskiplimit\maxdimen
    \m@th
    \ialign{##\crcr\hidewidth\smash{$-$}\hidewidth\crcr$\pitchfork$\crcr}
  }%
}
\makeatother

\def\s{\subseteq}
\def\forces{\Vdash}

\DeclareMathOperator{\cf}{cf}

\DeclareMathOperator{\ord}{Ord}
\DeclareMathOperator{\gHOD}{gHOD}

\newcommand{\dom}{\mathop{\mathrm{dom}}\nolimits}

\newcommand{\po}{\mathbb{P}}
\newcommand{\qo}{\mathbb{Q}}
\newcommand{\la}{\langle}
\newcommand{\ra}{\rangle}

\title[]{The number of  measures on very large measurable cardinals}
\author[Apter]{Arthur W. Apter}
\author[Kaplan]{Eyal Kaplan}
\author[Poveda]{Alejandro Poveda}

\address[Apter]{Department of Mathematics
Baruch College of CUNY
New York, New York 10010}
\email{awapter@alum.mit.edu}
\address[Kaplan]{Department of Mathematics, UC Berkeley, CA 94720, USA, \url{https://eyalkpl.github.io/website/}.}
\email{eyalkaplan@berkeley.edu}
\address[Poveda]{Universitat de Barcelona, Departament de Matemàtiques i Informàtica, Barcelona 08007, Catalonia, Spain \\
\url{https://alejandropovedaruzafa.com/}.}
\email{alejandro.poveda@ub.edu}


\begin{document}

\maketitle

\begin{abstract}
    We study the possible number of normal measures on a measurable cardinal in settings where inner model techniques are unavailable. Instead, we exploit consequences of the Ultrapower Axiom to obtain our theorems. We show that the classical Kimchi–Magidor result—that the first $n$ measurable cardinals can be strongly compact—can be combined with an arbitrary prescribed pattern for the number of normal measures they carry. We also prove that the first measurable cardinal above a supercompact cardinal can carry any given number of normal measures; the same conclusion is established for the first measurable limit of supercompact cardinals. As further applications of our techniques, we strengthen an unpublished theorem of Goldberg–Woodin and a theorem of Goldberg, Osinski, and Poveda. Our analysis circumvents both the reliance of Friedman–Magidor \cite{FriedmanMagidor} on core model methods and the limitations of the Prikry-type forcing iterations of \cite{GitikKaplan}. 
\end{abstract}

\section{Introduction}

The present manuscript studies the number of \emph{normal measures} on a \emph{measurable cardinal}. Measurable cardinals trace their origin to the classical work of S. Ulam \cite{Ulam} in connection with the Lebesgue measure problem. A cardinal $\kappa$ is called \emph{measurable} if there exists a non-principal, $\kappa$-complete ultrafilter on $\kappa$.\footnote{Such an ultrafilter $U$ trivially induces a nontrivial, two-valued, $\kappa$-additive measure on $\mathcal{P}(\kappa)$. Indeed, the measure $\mu\colon \mathcal{P}(\kappa)\to {0,1}$ associated with $U$ assigns the value $1$ to a set if and only if it belongs to $U$.} Normal measures on $\kappa$ are such ultrafilters that are additionally closed under diagonal intersections. It is well-known that every measurable cardinal carries at least one normal measure. The question of how many normal measures a measurable cardinal may carry has been a significant and well-studied problem in recent decades, culminating in its complete solution by Friedman and Magidor in \cite{FriedmanMagidor} (see the discussion below).
\smallskip

In this paper, we are interested in measurable cardinals that coexist with \emph{very large cardinals}, such as strongly compact and supercompact cardinals. A cardinal $\kappa$ is \emph{strongly compact} if, for every cardinal $\lambda \geq \kappa$, there exists a fine, $\kappa$-complete ultrafilter on $\mathcal{P}_{\kappa}(\lambda)=\{X\subseteq \lambda : |X|<\kappa\}$. A cardinal $\kappa$ is \emph{supercompact} if, for every $\lambda \geq \kappa$, there exists a fine, normal ultrafilter on $\mathcal{P}_{\kappa}(\lambda)$. Both are very strong forms of measurability and strictly exceed measurable cardinals in consistency strength.
\smallskip

Studying the higher levels of the large cardinal hierarchy can be challenging due to the current lack of a \emph{canonical inner model} for a supercompact cardinal. The construction of canonical inner models in set theory originated with Gödel's discovery of the universe of constructible sets, denoted by $L$. However, $L$ itself cannot accommodate most large cardinals—for instance, by a classical theorem of Scott, there are no measurable cardinals in $L$. Nevertheless, over the years a sequence of $L$-like models accommodating progressively stronger large cardinals has been developed, and the inner model program aims to construct such models for all large cardinal axioms.
\smallskip

The inner model theoretic approach provides powerful tools and techniques for addressing fundamental questions about large cardinals. The possible number of normal measures on a measurable cardinal is a concrete example of such a question. Indeed, Kunen constructed the canonical inner model for a measurable cardinal, $L[U]$, and showed that it carries a unique normal measure. Later, Mitchell developed inner models of the form $L[\vec{U}]$, constructed from coherent sequences of measures $\vec{U}$, which can exhibit measurable cardinals carrying any reasonable number of normal measures. However, Mitchell's construction requires assumptions beyond the existence of a single measurable cardinal. It was not clear whether the existence of a single measurable cardinal suffices to establish the consistency of a measurable cardinal $\kappa$ carrying exactly $\lambda$ normal measures, for a prescribed cardinal $1<\lambda<\kappa^{++}$.\footnote{We concentrate here on the GCH context, under which the number of normal measures on a measurable cardinal $\kappa$ is at most $\kappa^{++}$.}
\smallskip



The first major step towards resolving this problem was made in the early 70's by Kunen--Paris \cite{KunenParis}, who showed that starting from a single measurable cardinal $\kappa$, one can produce a forcing extension in which the number of normal measures on $\kappa$ is any cardinal $\lambda$ of cofinality greater than $\kappa^+$. In particular, one can achieve the maximal possible number of normal measures, $2^{2^\kappa}$. More recently, Apter--Cummings--Hamkins \cite{ApterCummingsHamkins} addressed the first case beyond the scope of the Kunen--Paris technique; namely, starting from a single measurable cardinal $\kappa$, the authors constructed a generic extension where  $\kappa$ carries exactly $\kappa^+$ normal measures.  Still, this left open the question of whether, from a single measurable cardinal, one can produce a model of ZFC with exactly $\lambda$ normal measures for any cardinal $\lambda \leq \kappa^+$. 

The above problem was completely resolved by Friedman and Magidor in their pioneering work \cite{FriedmanMagidor}. Friedman--Magidor's work not only settled a long-standing open problem, but also skillfully exploited a number of sophisticated techniques that nowadays play a pivotal role in the field. These include forcing over canonical inner models, the use of nonstationary support iterations, generalized Sacks forcing, and self-coding forcing. The lasting influence of Friedman--Magidor's ideas in the area can be better appreciated through   subsequent works like \cite{BenNeriaII, BenNeriaI, BenNeriaUnger, BenNeriaGitik, ApterCummingsTall, HP, aptercummings2023normalmeasuresonlargecards, GitikKaplan, BenNeriaKaplan, HP2}.

\smallskip

One of the distinctive aspects of the Friedman--Magidor proof is that it requires forcing over a canonical inner model (see \cite{MitHand}). This feature is essential to their strategy for controlling the normal measures that appear in the generic extension in two ways. First, the new normal measures satisfy that the restriction of their ultrapower embeddings to the ground model is an iterated ultrapower  of the core model  \cite{KunenMeasures, Mitchellcore, Schi} ––  such iterated ultrapowers are particularly well understood, thanks to the nature of the core model. Second, the Friedman--Magidor forcing incorporates a coding mechanism that prevents unwanted normal measures from being created, and in turn this coding utilizes the fine structure of the ground model

\smallskip

This fine-structural  approach, if elegant, imposes serious limitations on the extent of the applicability of Friedman--Magidor's methods. Namely, they are  limited to those large cardinals within the reach of current inner model theory.
For instance, the methodology cannot be leveraged to say anything about the number of normal measures on the first strongly compact cardinal, let alone on the first measurable limit of supercompact cardinals.

\smallskip

In response to this, work of Gitik and the second author \cite{GitikKaplan} suggested an alternative route to bypass these limitations. Their methods make it possible to show, for instance, that the first strongly compact cardinal can carry exactly $\eta$ normal measures $\langle U_\tau : \tau < \eta \rangle$, where each $U_\tau$ has Mitchell order $\tau$, and there are no other measurable cardinals with Mitchell order greater than or equal to $\eta$  \cite[Theorem 4.16]{GitikKaplan}. This is proved assuming the consistency of a supercompact cardinal with GCH and the linearity of the Mitchell order\footnote{In fact, weaker assumptions suffice; see \cite[Theorem 4.16]{GitikKaplan}.}. 

The proof technique involves performing a nonstationary support iteration of Prikry forcings (see \cite{BenNeriaUnger}), adding a Prikry sequence to every measurable cardinal of Mitchell order $\geq \eta$ below the least supercompact cardinal $\kappa$. Under the additional assumption that there are no measurable cardinals above $\kappa$, it is shown that $\kappa$ remains strongly compact in the generic extension, and that the normal measures on $\kappa$ have the desired configuration.

\smallskip

This method was further refined in \cite{GitikKaplan} to establish the consistency of the first strongly compact cardinal $\kappa$ being the least measurable cardinal and carrying exactly $\eta$ normal measures, for any given $\eta<\kappa$ (see \cite[Theorem 4.17]{GitikKaplan}). The argument begins with the model described above, in which the least strongly compact cardinal $\kappa$ carries $\eta$ normal measures, and proceeds by performing a Magidor (i.e., full support) iteration of Prikry forcings, adding a Prikry sequence to every measurable cardinal below $\kappa$. By a well-known theorem of Magidor  \cite{MagidorIdentity}, $\kappa$ is strongly compact and the least measurable cardinal in the resulting model. As observed by Ben-Neria in \cite{BenNeriaMagidorIterations}, the Magidor iteration of Prikry forcings does not change the number of normal measures on $\kappa$, and hence $\kappa$ carries exactly $\eta$ normal measures in the generic extension. While Ben-Neria assumes that the Magidor iteration is performed over a canonical inner model, his result remains valid under the weaker assumption of GCH in the ground model (see \cite{KaplanMagidorIterations}).

\smallskip

A natural question arising from the above is whether similar results can be obtained, for instance, at the second strongly compact cardinal. The methods of \cite{GitikKaplan} encounter two main obstacles. First, anti–large cardinal assumptions are used to ensure that a supercompact cardinal $\kappa$ remains strongly compact after a nonstationary support iteration below $\kappa$ –– it is assumed that there are no measurable cardinals above $\kappa$. Second, the least strongly compact cardinal $\kappa$ ceases to be strongly compact after a Prikry sequence is added above it, preventing the use of similar methods to control the number of normal measures on the second strongly compact cardinal.

\smallskip

In \cite{Kaplan}, the second author introduced the splitting forcing, which circumvents the reliance on inner model theory present in the Friedman--Magidor forcing.\footnote{The method has the additional virtue of bypassing the use of generalized Sacks forcing and self-coding forcings.} One of the principal advantages of the splitting forcing is its remarkably simple nature. This simplicity is largely due to suggestions of Omer Ben-Neria, whose insights significantly streamlined and improved the second author’s original forcing construction. In the present paper, we employ variations of the splitting forcing to bypass the use of Prikry-type forcings in \cite{GitikKaplan}. A key appeal of the splitting forcing is its flexibility: It can be modified, refined, and combined with other forcing techniques in a way that preserves many large cardinal configurations already present in the ground model.

\smallskip

Our theorems (see Theorems \ref{thm: IdentityCrises + Normal measures}, \ref{thm: above supercompact}, \ref{thm: first limit of supercompacts}, \ref{thm: Goldberg-Woodin}, and \ref{thm: Hod hypothesis} below) are proved relative to the assumption that the relevant large cardinals are consistent with GCH and Goldberg’s Ultrapower Axiom (UA) (see \cite{goldbergtheultrapoweraxiom} for details on UA). In fact, the proofs use only one consequence of UA: The existence of a unique normal measure of Mitchell order $0$ on every measurable cardinal.\footnote{This follows immediately from the linearity of the Mitchell order, which is a well-known consequence of UA (see \cite[Theorem 2.3.11]{goldbergtheultrapoweraxiom}).}  Thus, it suffices to assume that this property is consistent with the relevant large cardinal hypotheses. The plausibility of our initial assumptions is strongly supported by Goldberg’s work, which includes a detailed analysis of the structure of the universe under UA in the presence of very large cardinals.
\smallskip

Building upon these ideas, we prove the following:

\begin{manualtheorem}{\ref{thm: IdentityCrises + Normal measures}}
    Assume the $\mathrm{GCH}$ holds, there are $n<\omega$ supercompact cardinals $\langle \kappa_i: i<n\rangle$, there are no measurable cardinals above $\kappa_{n-1}$,  and each $\kappa_i$ has a unique normal measure of Mitchell order $0$. For every $i<n$, let $\tau_i\leq\kappa^{++}_i$ be a cardinal. Then there is a generic extension where {\rm{GCH}} holds, $\langle \kappa_i: i<n\rangle$ are the first $n$ strongly compact cardinals, the first $n$ measurable cardinals,  and each $\kappa_i$ carries exactly $\tau_i$ normal measures.
\end{manualtheorem}

\begin{manualtheorem}{\ref{thm: above supercompact}}
    Assume the $\mathrm{GCH}$ holds,  $\kappa$ is a supercompact cardinal, and $\lambda$ is the first measurable cardinal above $\kappa$. In addition, suppose that $\lambda$ has a unique normal measure. Then, for each cardinal $\tau\leq \lambda^{++}$, there is a forcing extension  where the following hold:
    \begin{enumerate}
        \item  $\kappa$ is supercompact;
        \item  $\lambda$ is measurable;
        \item $\lambda$ carries exactly $\tau$ normal measures.
    \end{enumerate}
\end{manualtheorem}

\begin{manualtheorem}{\ref{thm: first limit of supercompacts}}
  Assume the $\mathrm{GCH}$ holds, and  suppose that $\kappa$ is the first measurable limit of supercompact (strong) cardinals. Assume that $\kappa$ carries a unique normal measure of Mitchell order 0. Let $\tau\leq \kappa^{++}$ be a cardinal. Then, there is  a generic extension where $\kappa$ remains the first measurable limit of supercompact (strong) cardinals, and $\kappa$ carries exactly $\tau$ normal measures. 
\end{manualtheorem}

Theorem \ref{thm: IdentityCrises + Normal measures} combines the classical Kimchi--Magidor identity crisis theorem with our new techniques for analyzing the number of normal measures. An important component in the original Kimchi--Magidor proof is its use of Laver's forcing which makes the supercompact cardinals indestructible under sufficiently directed-closed forcings. However, Laver's indestructibility preparation also introduces many unwanted normal measures to the generic extension; our analysis includes a new version of the indestructibility theorem, which grants control over the number of normal measures of trivial Mitchell rank in the generic extension (see Theorem \ref{theorem: Laver indestructibility}). 

As mentioned above, anti-large cardinal assumptions were used in the Gitik--Kaplan proof that the least measurable cardinal may be strongly compact and carry a unique normal measure. In an unpublished work, the same result was already established by Goldberg--Woodin relative to the consistency of UA with a measurable limit of supercompact cardinals, without appealing to any anti--large cardinal assumption. Building on that, and relying on Theorem \ref{thm: first limit of supercompacts}, we are able to prove the following strengthening of the original Goldberg--Woodin theorem:

\begin{manualtheorem}{\ref{thm: Goldberg-Woodin}}
    Assume the $\mathrm{GCH}$ holds, $\kappa$ is a  measurable limit of supercompact cardinals, $\tau\leq \kappa^{++}$, and $\kappa$ has a unique normal measure of Mitchell order 0. Then there is a generic extension where $\kappa$ is the least measurable cardinal, the least strongly compact cardinal, and $\kappa$ carries exactly $\tau$ normal measures.
\end{manualtheorem}

\smallskip

We close the paper with an improvement of a theorem of Goldberg, Onsinski and Poveda 
\cite{GoldbergPoveda} on the status of the first supercompact cardinal in HOD, assuming Woodin's HOD hypothesis. Recall that under the HOD hypothesis, the first extendible cardinal (in fact, $\HOD$-supercompact suffices) is supercompact in HOD \cite{WooPartI, WooDavRod}.\footnote{This result is best possible, in the sense that the first extendible cardinal may consistently be the first strongly compact cardinal in $\HOD$ \cite{GoldbergPoveda}.} A natural question is if this is optimal -- namely, whether or not the first supercompact cardinal must be supercompact in HOD, under the HOD hypothesis.  This was answered in the negative in \cite{GoldbergPoveda} by producing a model of the HOD hypothesis where the first supercompact cardinal $\kappa$ is strongly compact yet not $2^\kappa$-supercompact in HOD. Taking the ideas of \cite{GoldbergPoveda} as a stepping stone, we  produce a model of the HOD hypothesis where the first supercompact cardinal is strongly compact in HOD, and it carries exactly one normal measure. This strengthens the configuration obtained in \cite{GoldbergPoveda}.

\begin{manualtheorem}{\ref{thm: Hod hypothesis}}
Assume $\mathrm{GCH}$ and $V=\gHOD$ both hold, $\kappa$ is a supercompact cardinal, and there exists a unique normal measure on $\kappa$ with trivial Mitchell rank. Then, the following configuration is consistent:
    \begin{enumerate}
        \item The $\mathrm{HOD}$ hypothesis holds.
        \item $\kappa$ is supercompact.
        \item In $\HOD$, $\kappa$ is strongly compact  and it carries exactly one normal measure. 
    \end{enumerate}    
\end{manualtheorem}

\section{Preliminaries}\label{sec: preliminaries}

\subsection{Generalities on nonstationary support iterations}
The main technique used in this paper is nonstationary support forcing iterations.

We leverage this section to  garner a few standard facts about these for later use.  Recall that a set of ordinals $A$ is called \emph{nonstationary in inaccessibles} if for every inaccessible cardinal $\lambda$, $A\cap \lambda$ is nonstationary in $\lambda$. 

\smallskip

Let $\kappa$ be a Mahlo cardinal. An iterated forcing $\po = \langle \po_\alpha, \dot{\qo}_\alpha  \colon \alpha<\kappa\rangle$ is a \emph{nonstationary support iteration} if for every inaccessible cardinal $\alpha\leq \kappa$, $\po_\alpha$ is the nonstationary support limit of $\la \po_\beta \colon \beta<\alpha \ra$, and for every other value of $\alpha<\kappa$, $\po_\alpha$ is the inverse limit of $\la \po_\beta \colon \beta<\alpha \ra$. More pedantically, for every $\alpha\leq \kappa$,  conditions $p\in \po_\alpha$ are functions with domain $\alpha$, such that:
\begin{enumerate}
    \item For all $\beta<\alpha$, $p\restriction \beta \in \po_\beta$ and $p\restriction \beta\Vdash p(\beta)\in \dot{\qo}_\beta$.
    \item The support of $p$, $$\textstyle \supp(p) =\alpha\setminus  \{  \beta<\alpha \colon p\restriction \beta \Vdash p(\beta) = \one_{\dot{\qo}_\beta} \}, $$ is nonstationary in inaccessibles.
\end{enumerate}
Assume that $\kappa$ is a Mahlo cardinal and $I\subseteq \kappa$ is a stationary set of inaccessible cardinals. We say that $\po$ is an \emph{$I$-spaced nonstationary support iteration} if, in addition to being a nonstationary support iteration $\la \po_\alpha, \dot{\qo}_\alpha \colon \alpha<\kappa \ra$, $\po$ satisfies that for every $\alpha\in \kappa\setminus I$, $\dot{\qo}_\alpha$ is forced by $\po_\alpha$ to be trivial forcing. 

In our context, $\mathbb{P}$ will be a nonstationary support iteration whose iterates $\dot{\mathbb{Q}}_\alpha$ possess certain closure properties, such as \emph{$\alpha$-closure} or  \emph{$\alpha$-strategic-closure}. Recall that a forcing poset $\mathbb{Q}$ is  \emph{$\alpha$-closed} if any decreasing sequence of conditions in $\mathbb{Q}$ with length less than $\alpha$ admits a lower bound in $\mathbb{Q}$. 
\smallskip

A natural weakening of $\alpha$-closure that is very much prevalent in the literature is the notion of $\alpha$-strategic closure:
\begin{definition}[$\alpha$-strategic closure]\label{def: modified strategic closure}\hfill
        \begin{enumerate}
            \item Let $\qo$ be a forcing notion and $\alpha$ be a regular cardinal. Let $G_\alpha(\qo)$ be the two player game consisting of $\alpha$ stages, in which the players construct a descending sequence of conditions. Player I plays at odd stages, and Player II plays at even stages (including limit stages). Player I wins if at some limit stage below $\alpha$, Player II fails to pick a condition extending all the conditions constructed so far. Otherwise, Player II wins.
            \item 
            We say that a forcing notion $\qo$ is \emph{$ \alpha$-strategically-closed} if Player II has a winning strategy in the game $G_{\alpha}(\qo)$.\footnote{This notion is also referred to as $\prec \alpha$-strategic-closure in work of the first author with Shelah \cite{ApterShelah}.} Similarly, $\mathbb{Q}$ is \emph{${<}\alpha$-strategically closed } if for all $\beta<\alpha$, Player II has a winning strategy in $G_\beta(\mathbb{Q})$.
        \end{enumerate}
    \end{definition}

\smallskip

An instrumental tool in the context of nonstationary support iterations is the so-called \emph{fusion lemma}. Roughly speaking, this lemma ensures (modulo some extra closure-type assumptions) that if $\mathbb{P}$ is a nonstationary support iteration and $\vec{d}=\langle d(\alpha) : \alpha <\kappa\rangle$ is a sequence of dense open sets of $\mathbb{P}$ then one can 'diagonalize' $\vec{D}$ on a club set of indices $\alpha$.  

\begin{lemma}[The fusion lemma]\label{Lemma: fusion lemma}
    Assume that $\kappa$ is Mahlo and $I\subseteq \kappa$ is a stationary set of inaccessible cardinals. Let $\la \po_\alpha, \dot{\qo}_\alpha \colon \alpha<\kappa \ra$ be an $I$-spaced nonstationary support iterated forcing, such that:
    \begin{enumerate}
        \item For every $\alpha<\kappa$, $\mathrm{rank}(\dot{\qo}_\alpha) < \min(I \setminus (\alpha+1))$.
        \item For every $\alpha< \kappa$, $\one \Vdash_{\po_\alpha} {``\dot{\qo}_\alpha \text{ is }\alpha\text{-strategically closed}}$".
        \item For every $\alpha\in \kappa\setminus I$, $\one \Vdash_{\po_\alpha} {``\dot{\qo}_\alpha \text{ is trivial forcing".}}$
       
    \end{enumerate}
    Let $\po$ be the nonstationary support limit of $\la \po_\alpha \colon \alpha<\kappa \ra$. Then for every $p\in \po$ and a sequence $\vec{d} = \la d(\alpha)\colon \alpha<\kappa \ra$ of dense open subsets of $\po$, there exists a condition $p^*\leq p$ and a club $C\subseteq \kappa$ such that for every $\alpha\in C$, 
    $$\{ r\in \po_{\alpha+1} \colon r^{\smallfrown} (p^*\setminus \alpha+1) \in d(\alpha) \}$$
    is dense in $\po_{\alpha+1}$ below $p^*\restriction (\alpha+1)$.
\end{lemma}

\begin{proof}
    Fix for every $\alpha\in I$ a $\po_\alpha$-name $\dot{\tau}_\alpha$ for a strategy of Player II in the game $G_{\alpha}( \dot{\qo}_\alpha)$.
    We construct sequences:
    \begin{itemize}
        \item $\la  p_i \colon i<\kappa\ra$ a decreasing  sequence of conditions in $\po$.
        \item $\la \alpha_i \colon i<\kappa \ra$ a continuous, increasing cofinal sequence in $\kappa$.
        \item $\la C_i \colon i<\kappa \ra$ a decreasing sequence of club subsets of $\kappa$ (with respect to inclusion), where each $C_i$ is disjoint from $\supp(p_i)$. 
    \end{itemize}
    The construction is done in such a way that the following properties hold:
    \begin{itemize}
        \item For every $i<j<\kappa$, $p_{i}\restriction \alpha_i+1 = p_j \restriction \alpha_i+1$.
        \item For every $i<\kappa$, $\{ r\in \po_{\alpha_{i}+1} \colon r^{\frown} p_{i}\setminus (\alpha_i+1) \in d(\alpha_i) \}$ is dense in $\po_{\alpha_i+1}$ below $p_i\restriction (\alpha_i+1)$.
        \item For every $i<j$, $\alpha_j \in C_i$. 
        \item For every $i<\kappa$ and $\alpha\in \supp(p_i)\setminus (\alpha_i+1)$, $p_i\restriction \alpha$ forces that $\la p_{j}(\alpha) \colon j<i, \alpha\in \supp(p_j) \ra$ is the sequence of moves of Player II in a partial run of the game $G_\alpha(\dot{\qo}_\alpha)$ in which Player II plays according to their winning strategy $\dot{\tau}_\alpha$.
    \end{itemize}
    
    Start the construction by letting $p_0 = p$, $\alpha_0 = 0$, and $C_0$ be any club disjoint from $\supp(p)$. Assume that $p_i, \alpha_i, C_i$ have been constructed for some $i<\kappa$. Let $\alpha_{i+1} = \min(C_i \setminus \alpha_i+1)$. We define $p_{i+1}\leq p_i$ such that:
    \begin{enumerate}
        \item $p_{i+1}\restriction \alpha_{i+1} = p_i \restriction \alpha_{i+1}$.
        \item $\alpha_{i+1}\notin \supp(p_{i+1})$ (and since $\alpha_{i+1}\in C_i$ and $C_i$ is disjoint from $\supp(p_i)$,  $p_{i+1}\restriction (\alpha_{i+1}+1) = p_{i}\restriction (\alpha_{i+1}+1)$).
        \item $p_{i+1}\setminus (\alpha_{i+1}+1)$ satisfies the following two properties:
        \begin{enumerate}
            \item $p_{i+1}\setminus (\alpha_{i+1}+1)$ is forced by $p_{i+1}\restriction (\alpha_{i+1}+1)$ to be an  extension $s\leq p_i\setminus (\alpha_{i+1}+1)$ for which $\{ r\in \po_{\alpha_{i+1}+1} \colon  r^{\frown} s\in d(\alpha_{i+1}) \}$ is a dense subset of $\po_{\alpha_{i+1}+1}$ below $p\restriction (\alpha_{i+1}+1)$.
            \item For every $\alpha\in \supp(p_{i+1})\setminus (\alpha_{i+1}+1)$, $p_{i+1}\restriction \alpha$ forces that $\la p_j(\alpha) \colon j\leq i, \alpha\in \supp(p_j) \ra$ is the sequence of moves of Player II in a run of the game $G_\alpha(\dot{\qo}_\alpha)$ in which Player II plays according to $\dot{\tau}_\alpha$.
        \end{enumerate}
    \end{enumerate}
\begin{claim}\label{Claim: fusion lemma, arranging clause 3}
    Clause~(3) can be arranged.
\end{claim}
\begin{proof}[Proof of claim]
   We need to argue that such a condition $s$ exists.  This essentially follows from the fact that $|\po_{\alpha_{i+1}+1}|$ is strictly below the amount of strategic closure of $\po\setminus (\alpha_{i+1}+1)$.\footnote{Indeed, the argument we provide below shows that for every $\alpha<\kappa$, the forcing $\po\setminus (\alpha+1)$ is $ \min(I\setminus (\alpha+1))$-strategically closed.} More formally, let $\langle q_\alpha\colon \alpha<\chi\rangle$ be an injective enumeration of all conditions in $\mathbb{P}_{\alpha_{i+1}+1}$ below $p_i\restriction (\alpha_{i+1}+1)$. We construct $\mathbb{P}_{\alpha_{i+1}+1}$-names $\langle \dot{p}_{i+1}^\beta\colon \beta\leq \chi\rangle$ as follows:
   \begin{itemize}
       \item $p_{i+1}\restriction (\alpha_{i+1}+1)$ forces that $\la \dot{p}^\beta_{i+1} \colon \beta\leq\chi \ra$ is a decreasing sequence of conditions in $\mathbb{P}\setminus (\alpha_{i+1}+1).$
       \item For every $\beta<\chi$, ${q_\beta}^{\frown} \dot{p}^{\beta}_{i+1}\in d(\alpha_{i+1})$.
       \item For every $\beta<\chi$ and $\alpha\in \supp(\dot{p}^\beta_{i+1})\setminus (\alpha_{i+1}+1)$, $\dot{p}^{\beta}_{i+1}\restriction \alpha$ forces that $\la \dot{p}^\gamma_{i+1}(\alpha) \colon \gamma<\beta, \alpha\in \supp(\dot{p}^\gamma_{i+1}) \ra$ is the sequence of moves of Player II in a run of the game $G_\alpha(\dot{\qo}_\alpha)$ in which player II plays according to the strategy $\dot{\tau}_\alpha$. 
   \end{itemize}
   Once such a sequence is constructed, we pick the desired $\po_{\alpha_{i+1}+1}$-name $\dot{s}$ to be a condition on $\po\setminus \alpha_{i+1}+1$ with $\supp(\dot{s}) = \bigcup_{\beta\leq \chi} \supp(\dot{p}^\beta_{i+1})$, as follows: Assume that $\alpha\in (\alpha_{i+1}, \kappa)$ and $\dot{s}\restriction \alpha$ has been constructed. Since $\dot{s}\restriction \alpha$ extends ${p}_{i}\restriction (\alpha_{i+1}+1,\alpha)$, it forces that $\la p_{j}(\alpha) \colon j\leq i, \alpha\in \supp(p_j)\ra$ is the sequence of moves of Player II in the game $G_\alpha(\dot{\qo}_\alpha)$ according to $\dot{\tau}_\alpha$. Assume that, in the same game, Player I picks the condition $\dot{p}^\chi_{i+1}(\alpha)$ at stage $i+1$ (and, if $\alpha\notin \supp(\dot{p}^\chi_{i+1})$, we assume that Player I picks $\one_{\dot{\qo}_\alpha}$). Let $\dot{s}(\alpha)$ be the condition forced by $\dot{s}\restriction \alpha$ to be the response of Player II when playing according to $\dot{\tau}_\alpha$. Since $\alpha$ is an inaccessible cardinal strictly above $\alpha_{i+1}$ (and in particular,  $\alpha>\chi$), the strategy $\dot{\tau}_\alpha$ ensures that Player II has a valid response at each stage in the run of the game described above (recall that $\dot{\tau}_\alpha$ witnesses that $\dot{\qo}_\alpha$ is $\alpha$-strategically closed). This concludes the definition of the condition $\dot{s}$. Since $\chi < \min(I\setminus (\alpha_{i+1}+1))$, the set $\supp(\dot{s})$ is forced to be nonstationary in inaccessibles of $I\setminus (\alpha_{i+1}+1)$, so $\dot{s}\in \po\setminus (\alpha_{i+1}+1)$. It is  not hard to verify that it satisfies the requirements in clause $(3)$ above.

   Thus, it remains to construct the sequence $\langle \dot{p}_{i+1}^\beta\colon \beta\leq \chi\rangle$.

   We first describe how the condition $\dot{p}^0_{i+1}$ is chosen. For that, pick a condition $q'\leq {q_0}^{\frown} (p_i\setminus (\alpha_{i+1}+1))$ such that $q'\in d(\alpha_{i+1})$. Define $\dot{p}^0_{i+1}$ such that $\supp(\dot{p}^0_{i+1}) = \supp(q')$, and, for every $\alpha\in \supp(\dot{p}^0_{i+1})$, $\dot{p}^0_{i+1}\restriction \alpha$ forces that $\dot{p}^0_{i+1}(\alpha)$ is a name for the response dictated to Player II by $\dot{\tau}_\alpha$ after Player I picked $q'(\alpha)$ on their first move in the game $G_\alpha(\dot{\qo}_\alpha)$. 

   Assume now that $\beta\leq \chi$ and $\la \dot{p}^\gamma_{i+1} \colon \gamma<\beta \ra$ have been defined. Let $\dot{t}\in \po\setminus (\alpha_{i+1}+1)$ be inductively constructed, as follows: For every $\alpha\in (\alpha_{i+1}, \kappa)$, if $\dot{t}\restriction \alpha$ was defined, it forces that $\dot{t}(\alpha)$ is the move dictated to Player II by the strategy $\dot{\tau}_\alpha$ in a run of the game $G_{\alpha}(\dot{\qo}_\alpha)$ in which the moves of Player II were $\la \dot{p}^\gamma_{i+1}(\alpha) \colon \gamma<\beta, \alpha\in \supp(\dot{p}^\gamma_{i+1}) \ra$. Note that the condition $\dot{t}\in \po\setminus (\alpha_{i+1}+1)$ constructed this way is forced by $p_{i+1}\restriction (\alpha_{i+1}+1)$ to be a lower bound of $\la  \dot{p}^\gamma_{i+1} \colon \gamma<\beta\ra$. Pick an extension $q'\leq (q_{\gamma})^{\frown} \dot{t}$ such that $q'\in d(\alpha_{i+1})$. Now let $\dot{p}^{\beta}_{i+1}\in \po\setminus (\alpha_{i+1}+1)$ be such that, for every $\alpha\in (\alpha_{i+1}+1, \kappa)$, $\dot{p}^\beta_{i+1}\restriction \alpha$ forces that $\dot{p}^\beta_{i+1}(\alpha)$ is the move dictated to Player II by $\dot{\tau}_\alpha$ in the following run of the game $G_\alpha(\dot{\qo}_\alpha)$: The game begins with the run in which the moves of Player II were $\la \dot{p}^\gamma_{i+1}(\alpha) \colon \gamma<\beta, \alpha\in \supp(\dot{p}^\gamma_{i+1}) \ra^{\frown} \dot{t}(\alpha)$. Then Player I replies with the condition $q'(\alpha)$, and the response of Player II in its next round according to $\dot{\tau}_\alpha$ is taken to be $\dot{p}^\beta_{i+1}(\alpha)$.  

   This concludes the inductive construction of  the sequence $\la \dot{p}^\beta_{i+1} \colon \beta\leq \chi \ra$, and therefore concludes also the construction of the condition $\dot{s}$.
 
\end{proof}
    Let us argue that the condition $p_{i+1}$ constructed this way indeed has a nonstationary support. Fix an inaccessible $\lambda>\alpha_{i+1}$. We claim that there exists a club in $\lambda$ that belongs to $V$, and is forced by $p_{i+1}\restriction (\alpha_{i+1}+1)$ to be disjoint from $\supp(p_{i+1}\setminus (\alpha_{i+1}+1))$. Indeed, there exists a $\po\restriction (\alpha_{i+1}+1)$-name $\dot{D}$ for a club in $\lambda$ which is disjoint from $\supp(p_{i+1}\setminus (\alpha_{i+1}+1)$. By the fact that $\po\restriction (\alpha_{i+1}+1)$ is small relative to $\lambda$ (and, in particular, is $\lambda$-c.c.), there exists a club $D^*\in V$ such that $p_{i+1}\restriction (\alpha_{i+1}+1)$ forces that $D^*$ is contained in $\dot{D}$; in particular, $D^*$ is disjoint from $\supp(p_{i+1})$.
    
    Finally, let $C_{i+1}\subseteq C_i$ be a club subset of $\kappa$ in $V$ which is  disjoint from the support of $p_{i+1}$. This concludes the successor step in the construction.

    For limit steps in the construction, assume that $\la p_j, \alpha_j , C_j \colon j<i \ra$ were constructed for some $i<\kappa$. Let $\alpha_i = \sup_{j<i}\alpha_j$. We define $p_i \in \po$ such that:
    \begin{enumerate}
     \item $p_i\restriction \alpha_{i} = \bigcup_{j<i} p_j \restriction (\alpha_j+1)$. Note that the union is increasing, and in the case where $\alpha_i$ is inaccessible, the sequence $\la \alpha_j \colon j<i \ra$ is a club in $\alpha_i$ disjoint from the $\supp(p_i\restriction \alpha_i)$. Thus,  $p_i\restriction \alpha_i\in \po_{\alpha_i}$.
        \item $\alpha_i \notin \supp(p_i)$. In particular, $p_i\restriction (\alpha_i+1 )\leq p_j \restriction (\alpha_i+1)$ for every $j<i$, since $\alpha_i$ lies outside $\supp(p_j)$  in that of a  limit point of $C_j$.
        \item $p_{i}\setminus (\alpha_{i}+1)$ satisfies the following two properties:
        \begin{enumerate}
            \item $p_i\setminus (\alpha_i+1)$ is forced by $p_i\restriction (\alpha_i+1)$ to be a common extension $\dot{s}$ of $\la p_j \setminus (\alpha_i+1) \colon j<i \ra$ for which $\{  r\in \po_{\alpha_i+1} \colon r^{\frown} \dot{s} \in d(\alpha_i) \}$ is a dense subset of $\po_{\alpha_i+1}$ below $p\restriction (\alpha_{i}+1)$.
            \item For every $\alpha\in \supp(p_{i}\setminus (\alpha_{i}+1))$, $p_{i}\restriction \alpha$ forces that $$\la p_j(\alpha) \colon j\leq i, \alpha\in \supp(p_j) \ra$$ is the sequence of moves of Player II in a run of the game $G_\alpha(\dot{\qo}_\alpha)$ in which Player II plays according to $\dot{\tau}_\alpha$.
        \end{enumerate}
        The proof that such $\dot{s}$ can be constructed is similar to the proof of Claim \ref{Claim: fusion lemma, arranging clause 3}. Note that for each coordinate $\alpha > \alpha_i$, the forcing $\dot{\qo}_\alpha$ is sufficiently strategically closed in order to pick local lower bounds.
    \end{enumerate}

    Finally, let $p^* =\bigcup_{i<\kappa} p_i\restriction (\alpha_i+1)$. {Clearly, $\supp(p^*)\cap \lambda$ is nonstationary in $\lambda$, for every inaccessible cardinal $\lambda<\kappa$.} Also, $\supp(p^*)$ is nonstationary in $\kappa$, as witnessed by the club  $C =\langle \alpha_i \colon i<\kappa \rangle$. To complete the proof of Lemma \ref{Lemma: fusion lemma}, we have to verify the following:
    \begin{claim}
     For every $\alpha\in C$, 
    $$\{ r\in \po_{\alpha+1} \colon r^{\smallfrown} (p^*\setminus \alpha+1) \in d(\alpha) \}$$
    is dense in $\po_{\alpha+1}$ below $p^*\restriction (\alpha+1)$.
    \end{claim}
    \begin{proof}[Proof of claim]
        Let $\alpha\in C$. Our inductive construction ensures that $$D=\{r\in \mathbb{P}_{\alpha+1}\colon r\leq p_\alpha\restriction (\alpha+1)\,\wedge\, r{}^\smallfrown p_{\alpha}\setminus (\alpha+1)\in d(\alpha)\}$$
        is dense below $p_{\alpha}\restriction (\alpha+1)=p^*\restriction(\alpha+1).$
        
        Let $r\in D$ be arbitrary. Since $p^*\leq p_\alpha$, in particular,  $$p^*\restriction (\alpha+1)\forces p^*\setminus (\alpha+1)\leq p_\alpha\setminus (\alpha+1).$$
        But then $r{}^\smallfrown p^*\setminus (\alpha+1)\leq r{}^\smallfrown p_\alpha\setminus (\alpha+1)$, which yields
        $$D\s \{r\in \mathbb{P}_{\alpha+1}\colon r\leq p^*\restriction(\alpha+1)\,\wedge\, r{}^\smallfrown p^*\setminus (\alpha+1)\in d(\alpha)\},$$
        and as a result the latter is dense below $p^*\restriction (\alpha+1).$
    \end{proof}
    This concludes the proof of Lemma \ref{Lemma: fusion lemma}.
\end{proof}

\begin{lemma}\label{Lemma: nonstatsupport iterations preserving GCH}
    Let $\po = \la \po_\alpha, \dot{\qo}_\alpha  \colon \alpha<\kappa \ra$ be an $I$-spaced nonstationary support iteration satisfying the hypotheses of Lemma \ref{Lemma: fusion lemma}. Assume {\rm{GCH}} and suppose that for every $\alpha\in I$, $\dot{\qo}_\alpha$ is forced by $\po_\alpha$ to be a forcing notion that preserves cardinals. Then, the following hold:
    \begin{enumerate}
        \item \label{clause: nonstat iterations are nice 1} $\po$ preserves cardinals.
        \item \label{clause: nonstat iterations are nice 2} Suppose that, in addition, for each $\alpha\in I$, $\dot{\qo}_\alpha$ is forced by $\po_\alpha$ to be a forcing notion that preserves {\rm{GCH}}. Then $\po$ preserves {\rm{GCH}}.
    \end{enumerate}
\end{lemma}

\begin{proof}
    We first prove clause (1). By GCH, $|\po| = \kappa^+$. In particular, for every $\lambda\geq \kappa^{++}$, $\lambda$ remains a cardinal. The fact that $\kappa^+$ remains a cardinal follows from the following claim, which relies on the previous fusion lemma:
    \begin{claim}
      Let $\dot{f}$ be a $\mathbb{P}$-name such that $\one \forces_{\mathbb{P}} ``\dot{f}\colon\kappa\rightarrow \kappa^+$ is an increasing function".   Then, the set $\{p\in \mathbb{P}:  \exists \alpha<\kappa^+\,(p\forces_{\mathbb{P}}\mathrm{Im}(\dot{f})\s \alpha)\}$ is dense.

      In particular,    forcing with $\po$ preserves $\kappa^+$.
    \end{claim}
\begin{proof}[Proof of claim]
Let $\langle d(\alpha): \alpha<\kappa\rangle$  be the sequence of dense open subsets of $\po$, where each $d(\alpha)$ consists of all the conditions $p\in \mathbb{P}$ which decide the value of $\dot{f}(\alpha)$. By Lemma~\ref{Lemma: fusion lemma}, every condition can be extended to a condition $p$ for which there exists a club $C\subseteq \kappa$ such that, for every $\alpha\in C$,
    $$ \{  r\in \po_{\alpha+1} \colon r^{\frown} (p\setminus \alpha+1)\in d(\alpha)  \} $$
    is dense below $p\restriction (\alpha+1)$. In particular, let
    $$\gamma = \sup\{  \beta< \kappa^+ \colon  \exists \alpha\in C \ \exists r\in \po_{\alpha+1} \ \left( r^{\frown} (p\setminus \alpha+1)\Vdash \dot{f}(\alpha)< \beta \right) \}.$$
    Then $\gamma<\kappa^+$ (because $|\mathbb{P}_{\alpha+1}|<\kappa$) and clearly $p\Vdash \mathrm{Im}(\dot{f})\subseteq \gamma$.
\end{proof}

    In order to show that $\po$ preserves $\kappa$, it suffices to show that $\po$ preserves all cardinals $\lambda<\kappa$. Assume otherwise, and let $\lambda<\kappa$ be the least cardinal being collapsed by $\po$. Factor $\po = \po_{\lambda} * \dot{\qo}_{\lambda} * \po\setminus (\lambda+1)$. Since $\dot{\qo}_\lambda$ preserves cardinals and $\po\setminus (\lambda+1)$ is $\lambda^+$-distributive (see the proof of Claim \ref{Claim: indestructibility proof, claim 1}), the forcing $\po_\lambda$ collapses $\lambda$. By the argument from the previous paragraph, $\lambda$ cannot be a successor cardinal; this implies that $\lambda$ is a limit cardinal, contradicting the assumption that it is the least cardinal collapsed by $\po$.

    Let us proceed to clause \eqref{clause: nonstat iterations are nice 2}.
    Let $\lambda$ be a cardinal\footnote{Note that by the previous observation there is no ambiguity here.} and factor the iteration as $\mathbb{P}=\mathbb{P}_\lambda\ast \dot{\mathbb{Q}}_\lambda\ast \mathbb{P}\setminus (\lambda+1).$ In the ground model, $``2^\lambda=\lambda^+$" holds, so if $\mathbb{P}_\lambda$ were to preserve this fact we will be done – indeed, $\dot{\mathbb{Q}}_\lambda$ is forced to preserve the $\mathrm{GCH}$ and $\mathbb{P}\setminus (\lambda+1)$ is $\lambda^+$-distributive, ergo none of them change the power-set-function pattern at $\lambda$ from the ground model. We prove that  $\mathbb{P}_\lambda$ preserves $``2^\lambda=\lambda^+$" by induction on $\lambda$. Suppose that for every $V$-cardinal $\beta<\lambda$ forcing with $\mathbb{P}_\beta$ preserves the GCH at $\beta$. We will show that given any generic $G_\lambda\s \mathbb{P}_\lambda$ and a set $X\in \mathcal{P}^{V[G_\lambda]}(\lambda)$ there is a club $C\in \mathrm{Cub}^V_\lambda$ and a sequence $\langle X_\beta\mid \beta\in C\rangle\in \prod_{\beta\in C} \mathcal{P}^{V[G_\beta]}(\beta)$ 
    such that $X=\bigcup_{\beta\in C} X_\beta$.\footnote{In the case where $\lambda$ is $V$-singular, $C$ will instead be the ordinal $\cf^V(\lambda)$.} Note that this set of sequences has size $(2^\lambda)^V=\lambda^+$: 
    To show this notice that each sequence $\langle X_\beta\mid \beta\in C\rangle$ comes from a  (partial) $\lambda$-sequence of $\mathbb{P}_\beta$-nice names $\tau_\beta$, each  naming a subset of $\beta$. Since for each of those $\beta$,   $\tau_\beta\in V_\lambda$, we conclude that there are at most $|{}^\lambda V_\lambda|^V=\lambda^+$-many such sequences of names, ergo at most $\lambda^+$-many sequences $\langle X_\beta\mid \beta\in C\rangle$, and therefore at most $\lambda^+$-many $X\in \mathcal{P}^{V[G]}(\lambda)$.\footnote{In the case where $\lambda$ is $V$-singular we will have a $\cf^V(\lambda)$-sequence of members of $V_\lambda$, which  either way  has size $\lambda^+$, by the GCH in the ground model.}

    \smallskip

   \textbf{\underline{Case $\lambda$ is regular:}} 
  Let $X\in \mathcal{P}^{V[G_\lambda]}(\lambda)$. Without loss of generality let us assume that this is forced by the trivial condition. For each $\beta<\lambda$, let $$d(\beta):=\{p\in \mathbb{P}_\lambda :  p\restriction \beta+1\forces_{\mathbb{P}_{\beta+1}} \exists \dot{X}_\beta\s \beta\;(p\setminus (\beta+1)\forces_{\mathbb{P}_\lambda/\dot{G}_{\beta+1}} \dot{X}\cap \beta=\dot{X}_\beta)\}.$$ 
  This set is clearly open, and it is dense in that $\mathbb{P}_\lambda/\dot{G}_{\beta+1}$ is forced to be $\beta^+$-strategically-closed, by our assumptions in the lemma.

   Invoking Lemma~\ref{Lemma: fusion lemma}, we find a condition $p^*$ and a club $C\s \alpha$ such that, for each $\beta\in C$, the set $e(\beta):=\{r\in \mathbb{P}_{\beta+1}\mid r{}^\smallfrown p^*\setminus (\beta+1)\in d(\beta)\}$ is dense below $p^*\restriction (\beta+1)$. Since there are $\mathbb{P}_\lambda$-dense many such conditions we can assume that $p^*\in G_\lambda$. Now, for each $\beta\in C$ and $r\in e(\beta)$, we invoke the Forcing Maximality Principle to let $\dot{X}_{\beta,r}$ a  $\mathbb{P}_{\beta+1}$-name witnessing that $r\in e(\beta)$. Now we amalgamate those names into a single $\mathbb{P}_{\beta+1}$-name, taking
   $$\dot{X}_\beta:=\{\langle \dot{X}_{\beta,r},r\rangle: r\in e(\beta)\}.$$

   Clearly, $p^*\forces_{\mathbb{P}_\lambda}\dot{X}_\beta\s \dot{X}\cap \beta$.\footnote{Of course, here we are identifying $\dot{X}_\beta$ with its natural lifting to a $\mathbb{P}_\lambda$-name.} To show the converse, let $\gamma<\beta$ and $s\leq p^*$ forcing $``\gamma\in \dot{X}\cap \beta$". By density of $e(\beta)$ we may assume that $s\restriction (\beta+1)\in e(\beta)$. Therefore, $(s\restriction \beta+1){}^\smallfrown p^*\setminus (\beta+1)\in d(\beta)$ and thus the latter necessarily forces $``\gamma\in \dot{X}_{\beta,s\restriction(\beta+1)}\s \dot{X}_\beta$", hence so does the stronger condition $s$ as well. By a density argument this shows that  $p^*\forces_{\mathbb{P}_\lambda} \dot{X}\cap \beta\s \dot{X}_\beta$, as it was needed.
   \medskip

   \underline{\textbf{Case $\lambda$ is singular in $V$:}} 
   We consider  the case where $\lambda$ is a limit of members of $I\cap \lambda$ -- in particular, $\lambda$ is  strong limit in $V$. The case where $\lambda$ is not a limit of members of $I\cap \lambda$ is easier since then $\po_\lambda$ is forcing equivalent to the poset $\po_\gamma$ where $\gamma = \sup(I\cap \lambda)< \lambda$, and by a usual nice names argument and GCH in $V$ this latter adds at most $\lambda^+$-many subsets to $\lambda$.
    
    Denote $\zeta = \cf(\lambda)$. Consider the forcing $\po_\lambda\setminus (\zeta+1)$. It is $\zeta^{+}$-strategically closed. Fix $\la \lambda_\alpha \colon \alpha<\zeta \ra$ an increasing cofinal sequence in $\lambda$ such that $\lambda_0 > \zeta$. For each $\alpha<\zeta$, working in $V[G_\lambda]$, let us consider
    $$d(\alpha) = \{  q\in \po_\lambda/G_{\lambda_\alpha+1} \colon \exists A\subseteq \lambda_\alpha  (q\Vdash_{\mathbb{P}_\lambda/G_{\lambda_\alpha+1}} \dot{X}\cap\lambda_\alpha = \dot{A}) \}.$$
    Each $d(\alpha)$ is a dense open subset of $\po_\lambda/G_{ \lambda_\alpha+1}$, ergo the set 
    $$ D(\alpha) = \{  r\in \po_\lambda/G_{\zeta+1} \colon r\restriction \lambda_\alpha +1 \Vdash \exists A_\alpha\subseteq \lambda_\alpha  ( r\setminus \lambda_\alpha+1 \Vdash \dot{X}\cap \lambda_\alpha = A_\alpha) \} $$
    is dense open in $\po_\lambda/G_{\zeta+1}$. This is true for every $\alpha<\zeta$. 
    
    Since $\po_\lambda/G_{\zeta+1}$ is $\zeta^{+}$-strategically closed, the intersection $D^* = \bigcap_{\alpha<\zeta} D_\alpha$ is dense in $\po_\lambda/G_{\zeta+1}$. Work in $V[G_{\zeta+1}]$. If $q\in \po_\lambda/G_{\zeta+1}$ is a condition in the intersection $D^*$ then $q$ carries a sequence $\la \dot{A}_\alpha \colon \alpha<\zeta\ra$ such that each $\dot{A}_\alpha$ is a $\po_{\lambda_\alpha+1}/G_{\zeta+1}$-name for a subset of $\lambda_\alpha$, and $q$ basically reduces the name $\dot{X}$ to the sequence $\langle \dot{A}_\alpha \colon \alpha<\zeta \rangle$; to wit, $q\forces_{\mathbb{P}_\lambda/G_{\zeta+1}}\dot{A}_\alpha=\dot{X}\cap \lambda_\alpha$. 
    
    The above shows that the number of subsets of $\lambda$ in $V[G_\lambda]$ is bounded by the number of sequences $(V_\lambda)^{\zeta}$ in $V[G_{\zeta+1}]$. Standard $\mathbb{P}_{\zeta+1}$-nice names arguments shows that the cardinality of this set is bounded by
    $$\left(|(V_\lambda)^{\zeta}|^{|\mathbb{P}_{\zeta+1}|}\right)^V\leq |{}^{<\lambda}V_\lambda|^V=|\lambda^{\cf(\lambda)}|^V=\lambda^+.$$
    This completes the proof of the lemma.
   
\end{proof}
For later analysis we also record the proof of the following technical fact:
\begin{lemma}\label{lemma: capturing functions via fusion}
    Let $\po = \la \po_\alpha, \dot{\qo}_\alpha  \colon \alpha<\kappa \ra$ be an $I$-spaced nonstationary support iteration satisfying the hypotheses of Lemma \ref{Lemma: fusion lemma}. Let $\dot{f}\colon \kappa\rightarrow \ord$ be a $\mathbb{P}$-name for a function as forced by some condition $p\in \mathbb{P}$. Then, there is $p^*\leq p$, a club $C\s\kappa$ and a function $F\colon \kappa\rightarrow V$ with $|F(\alpha)|\leq |\mathbb{P}_{\alpha+1}|^V$ for every $\alpha\in C$, such that $p^*\forces \forall \alpha\in \check{C}\, (\dot{f}(\alpha)\in \check{F}(\alpha)).$  
\end{lemma}
\begin{proof}
  For each $\alpha<\kappa$ let $d(\alpha):=\{q\in \mathbb{P}\mid q\perp p\;\vee\; (\exists\beta\; q\forces \dot{f}(\alpha)=\beta) \}$.  This is a dense open subset of $\mathbb{P}$. By the Fusion Lemma \ref{Lemma: fusion lemma} there is $p^*\leq^* p$ and a club $C\s \kappa$ such that $e(\alpha):=\{r\in \mathbb{P}_{\alpha+1}\mid r{}^\smallfrown p^*\setminus (\alpha+1)\in d(\alpha)\}$ is dense below $p^*\restriction (\alpha+1)$ for all $\alpha\in C$. Consider $F\colon C\rightarrow V$ defined as $$F\colon \alpha\mapsto \{\beta_r\mid r\in e(\alpha),\,r\leq p^*\restriction (\alpha+1),\,r{}^\smallfrown p^*\setminus (\alpha+1)\forces \dot{f}(\alpha)=\beta_r\}.$$
  It is clear that $F$ is as wanted.
\end{proof}

Recall that an elementary embedding $j\colon V\rightarrow M$ is called an \emph{extender ultrapower embedding} if there is an extender $E$ such that $j$ is the ultrapower embedding induced by $E$, $j_E$. For an overview of the general theory of extenders we refer the reader to Kanamori's book  \cite{Kan}. An important consequence of Lemma \ref{Lemma: fusion lemma}  is the lifting of extender ultrapower embeddings:
\begin{lemma}\label{Lemma: lifting elementary embeddings with fusion}
    Let $I\s \kappa$ be a stationary set consisting of inaccessible cardinals and  $j\colon V\to M$ be an extender ultrapower embedding with $\crit(j) = \kappa$ for which the following properties hold:
    \begin{itemize}
        \item[$(\aleph)$] $V\vDash {}^\kappa M\subseteq M$.
        \item[$(\beth)$] The generators of $j$ are bounded below $j(h)(\kappa)$ for some function $h\colon \kappa\to \kappa$, which satisfies that $j(h)(\kappa)<\min(j(I)\setminus \kappa+1)$.
    \end{itemize}
    Let $\po = \la \po_\alpha, \dot{\qo}_\alpha \colon \alpha<\kappa \ra$ be an $I$-spaced nonstationary support iteration of length $\kappa$ satisfying the assumptions of Lemma \ref{Lemma: fusion lemma}. Let $G\subseteq \po$ be generic over $V$. Denote $\dot{\qo}_\kappa = j\left( \la \dot{\qo}_\alpha \colon \alpha<\kappa \ra \right)(\kappa)$. Then:
    \begin{enumerate}
        \item \label{Clause: G is generic over M} ${j(\po)}_\kappa= \po$, and $G\subseteq \po$ is generic over $M$.
        \item \label{Clause: g is generic} In $V[G]$, let $\qo_\kappa = (\dot{\qo}_\kappa)_G$. If $g\subseteq \qo_\kappa$ is generic over $V[G]$, then $g$ is $\qo_\kappa$-generic over $M[G]$.
        \item \label{Clause: j[G] generates a generic} $j[G]\setminus \kappa+1$ generates a $j(\po)\setminus \kappa+1$-generic set over $M[G*g]$; to wit, 
    $$H= \{ q \in j(\po)\setminus \kappa+1 \colon \exists p\in G \ ( q\geq j(p)\setminus \kappa+1 )\}.$$
  
    \end{enumerate}
    In particular, $j\colon V\to M$ lifts to an embedding $j^*\colon V[G]\to M[G*g*H]$ which is definable in $V[G*g]$. Moreover, if for each $\alpha\in I$ the forcing $\dot{\qo}_\alpha$ is forced to be either  $\alpha^{+}$-c.c. or $\alpha^+$-closed, $M[G*g*H]$ is closed under $\kappa$-sequences in $V[G*g]$.
   
\end{lemma}

\begin{proof}
    For clause \eqref{Clause: G is generic over M}, note that, for every $\alpha<\kappa$, $\po_\alpha\in V_{\kappa}$. Therefore $j(\po)_{\kappa}$ is, over $M$, the nonstationary support limit of $\la \po_\alpha \colon \alpha<\kappa \ra$. The nonstationary support limit up to $\kappa$ is computed correctly in $M$ since $M$ is closed under $\kappa$-sequences of its elements living in $V$.

    \smallskip
    
    Clause \eqref{Clause: g is generic} holds since $M[G]\subseteq V[G]$. 

    \smallskip
    
    For clause \eqref{Clause: j[G] generates a generic}, 
    assume that $E\subseteq j(\po)\setminus \kappa+1$ is a $\po_{\kappa}*\dot{\qo}_\kappa$-name for a dense open subset of $j(\po)\setminus \kappa+1$. Write $E = j(f)(a)$ for some set of generators $a\in [j(h)(\kappa)]^{<\omega}$ and some $f\colon [\kappa]^{|a|}\to V$. We can assume that $\kappa = \min(a)$ and, for every $\vec{\nu}\in \dom(f)$, $f(\nu)$ is a $\po_{\min (\vec{\nu})+1}$-name for a dense open subset of $\po\setminus ( \min(\vec{\nu})+1 )$. Define, for every $\alpha<\kappa$, 
    $$d(\alpha) = \{ q\in \po \colon q\restriction \alpha \Vdash \forall \vec{\nu}\in [h(\alpha)]^{|a|} \ ( q\setminus \alpha+1\in f(\vec{\nu})) \}.$$
    We claim that $d(\alpha)$ is a dense open subset of $\po$. For that, it suffices to prove that $\po\setminus (\alpha+1)$ is forced to be $|h(\alpha)|^+$-distributive. Indeed, $\po\setminus (\alpha+1)$ is forced to be  $\min(I)\setminus (\alpha+1)$-strategically closed; this follows from the combination of the facts that the iteration is taken with respect to an $I$-spaced nonstationary support, and for each active forcing stage $\beta\in I\setminus (\alpha+1)$, $\dot{\qo}_\beta$ is forced to be $ \beta$-strategically closed (see the proof of Claim \ref{Claim: fusion lemma, arranging clause 3} for a similar argument).

    By the Lemma \ref{Lemma: fusion lemma}, there exists $p\in G$ and a club $C\subseteq \kappa$ such that, for every $\alpha\in C$,
    $$\{r\in \mathbb{P}_{\alpha+1}\colon r{}^\smallfrown p\setminus (\alpha+1)\in d(\alpha)\}$$
    is dense below $p\restriction (\alpha+1)$. Clearly this yields
    $$p\restriction \alpha+1 \Vdash \forall \vec{\nu}\in [h(\alpha)]^{|a|} \ ( p\setminus (\alpha+1)\in f(\vec{\nu})).  $$
    Since $\po$ is the nonstationary support limit of $\la \po_\alpha \colon \alpha<\kappa \ra$, we can assume that the club $C$ is disjoint from $\supp(p)$, by shrinking it if necessary. Since $C\subseteq \kappa$ is a club, we have $\kappa\in j(C)$. Therefore, $j(p)\restriction \kappa+1 = p^{\smallfrown} \one_{\qo_\kappa}\in G*g$, and thus, in $V[G*g]$, 
    $$j(p)\setminus \kappa+1 \in j(f)(a)_{G\ast g} = (E)_{G*g}$$
    as desired.

    Since every $p\in G$ has a nonstationary support and $H$ is generated by $j[G]\setminus \kappa+1$, we have that $j[G]\subseteq G*g*H$. By Silver's lifting criterion, $j$ lifts to an elementary embedding $j^*\colon V[G]\to M[G*g*H]$.

    Finally, let us assume that for every $\alpha\in I$, $\Vdash_{{\po}_\alpha} {`` \dot{\qo}_\alpha \text{ is }\alpha^+-c.c. "}$. We argue that $M[G*g*H]$ is closed under $\kappa$-sequences in $V[G*g]$. 

    \begin{claim}\label{Claim: closure under kappa sequneces and fusion}
        $M[G]$ is closed under $\kappa$-seuqneces in $V[G]$.
    \end{claim}

    \begin{proof}
    Assume that $f\colon \kappa\to \text{Ord}$ is a sequence in $V[G]$. Let $\dot{f}$ be a $\po$-name for it. By Lemma \ref{Lemma: fusion lemma}, there exists $p\in G$ and a club $C\subseteq \kappa$ such that for every $\alpha\in C$,
    $$p\restriction_{\alpha+1} \Vdash \exists g_\alpha \in ({Ord})^\alpha \  \left( p\setminus \alpha+1 \Vdash f\restriction \alpha = g_\alpha \right).$$
    This is due to the fact that the forcings $\po\setminus (\alpha+1)$ are $\alpha^+$-strategically closed, so the sets
    $$ d(\alpha) = \{ q\in \po\setminus \alpha+1 \colon q \text{ decides } \dot{f}\restriction \alpha \}  $$
    are dense and open. 

    Next, for every $\alpha\in C$, fix a $\po_{\alpha+1}$-name $\dot{g}_\alpha$, forced by $p\restriction_{\alpha+1}$ to be the set $g_\alpha\in V^{\po_\alpha}$ above. Since $M$ is closed under $\kappa$-sequences in $V$, $\langle \dot{g}_\alpha \colon \alpha\in C \rangle\in M$. Thus, $(\dot{f})_G$ can be reconstructed inside $M[G]$ from $G$ and $\langle \dot{g}_\alpha \colon \alpha\in C \rangle$ as the union $$\bigcup_{\alpha\in C} (g_\alpha)_{G\restriction \alpha+1}.$$
    This concludes the proof of the claim.
    \end{proof}

    Next, we argue that $M[G*g]$ is closed under $\kappa$-sequences in $V[G*g]$. This is clear if $\qo_\kappa$ is $\kappa^+$-c.c. or $\kappa^+$-closed in $V[G]$ (for the $\kappa^{+}$-c.c. case, see \cite[Proposition 8.4]{CummingsHandbook}). In turn, either of these properties hold true in $V[G]$ because $M[G]$ is closed under $\kappa$-sequences in $V[G]$, and the fact that $\qo_\kappa$ is either $\kappa^+$-c.c. or $\kappa^+$-closed in $M[G]$ by our assumption. 

    Finally, the fact that $M[G*g*H]$ is closed under $\kappa$-sequences in $V[G*g]$ follows immediately from the fact that $M[G*g]$ is closed under $\kappa$-sequences in $V[G*g]$.\end{proof}

To conclude this preliminary section, we mention two theorems that will be key to our forcing analysis. The first is Laver’s \emph{Ground Model Definability Theorem}, and the second is Hamkins’ \emph{Gap Forcing Theorem}.

\begin{theorem}[Laver's Ground Model Definability Theorem \cite{Lavergroundmodeldefinability}]\label{thm: Laver ground model definability thm}
    Suppose that $\po\in V$ is a forcing notion and $G\subseteq \po$ is generic over $V$. Then $V$ is a definable class of $V[G]$ from a parameter in $V$.
\end{theorem}

\begin{theorem}[Hamkins' Gap Forcing Theorem \cite{HamkinsGap}]\label{thm: Gap forcing theorem}
    Suppose that $\po$ is a forcing notion which has a gap at some cardinal $\delta$; namely, $\po$ can be factored in the form $\po = \po_0*\dot{\po}_1$, where $\po_0$ is nontrivial,  $|\po_0|<\delta$ and $\Vdash_{\po_0} ``\dot{\po}_1 \text{ is }(\delta+1)\text{-strategically closed."}$ Let $G\subseteq \po$ be generic over $V$, and assume that $j^*\colon V[G]\to M^*$ is an elementary embedding with critical point $\kappa > \delta$, such that $M^*\subseteq V[G]$ and $M^*$ is closed under $\delta$-sequences of its elements that belong to $V[G]$. Then:
		\begin{enumerate}
			\item $M^*$ has the form $M[H]$, where $M\subseteq V$ and $H = j^*(G)$ is $j^*(\po)$-generic over $M$. Furthermore, $M = V\cap M[H]$.
			\item If $j^*$ is definable in $V[G]$ from parameters, then $j = j^*\restriction V\colon V\to M$ is definable in $V$ from parameters. \end{enumerate}
\end{theorem}

\subsection{The splitting forcing}\label{Section: splitting forcing}
		In this section we outline, for later use, a variation of the \emph{splitting forcing} \label{page: splitting forcing} introduced by the second author in \cite{Kaplan}. 
        
        For the rest of this section, we fix a measurable cardinal $\kappa$, and for each $\eta<\kappa^{+}$ let $f_\eta \colon \kappa \to \kappa$ be the $\eta$-th canonical function.\footnote{For the explicit definition of $\la f_\eta \colon \eta<\kappa^+ \ra$ we refer the reader to \cite[Lemma 24.5]{Jech}.} To  streamline notation, we also denote by $f_{\kappa^+}$ the  function $f_{\kappa^+}\colon \alpha \mapsto |\alpha|^+$.

\smallskip

Fix an ordinal $\tau<\kappa^+$ and $I\s \kappa$ a stationary set consisting of inaccessible cardinals. Define an $I$-spaced nonstationary support iteration $$\po^{\tau, I} = \la \po^\tau_\alpha, \dot{\qo}^\tau_\alpha \colon \alpha<\kappa \ra,$$ where, for each $\alpha<\kappa$, we distinguish between the following cases:
\begin{itemize}
    \item If $\alpha\in I$, $\po^\tau_\alpha$ forces that $\dot{\qo}^\tau_\alpha$ is an atomic forcing of the form $$\textstyle \{ \one_{\qo_\alpha} \}\cup f_\tau(\alpha),$$ where the ordinals in $f_{\tau}(\alpha)$ are declared to be an antichain in  $\dot{\mathbb{Q}}^\tau_\alpha$.
    \item Else, $\dot{\qo}^\tau_\alpha$ is forced by $\mathbb{P}_\alpha^\tau$ to be trivial. 
\end{itemize}

Given $G\subseteq \po^{\tau. I} $ generic over $V$, $\bigcup G$ naturally defines a function in  $\prod_{\alpha<\kappa} f_\tau(\alpha)$, which maps each $\alpha<\kappa$ to the unique common value $p(\alpha)\in f_{\tau}(\alpha)$ for conditions $p\in G$ such that $\alpha\in \supp(p)$.
Note that $\mathbb{P}^{\tau,I}$ satisfies (1)--(3) of Lemma~\ref{Lemma: fusion lemma}. In particular, Lemma \ref{Lemma: lifting elementary embeddings with fusion} will be applicable to $\po^{\tau,I}$  whenever we are presented with an appropriate extender ultrapower $j\colon V\rightarrow M$. Also,   $\po^{\tau,I}$ has a gap at any large enough cardinal that is a  successor of an inaccessible caridnal below $\kappa$.

\smallskip

The following theorem is a variation on one of the main results of \cite{Kaplan}. 

\begin{theorem}
\label{thm: properties of the splitting forcing}
    Assume that $\kappa$ is a measurable cardinal and $I\subseteq \kappa$ is a stationary set of inaccessible cardinals. Let $\tau\leq \kappa^+$ and $\po = \po^{\tau, I}$ be the above forcing. 
    Then the following hold:
    \begin{enumerate}
        \item $\po$ preserves both $\kappa$ and $ \kappa^+$ as  cardinals and it also preserves the measurability of $\kappa$.  Assuming $\mathrm{GCH}$, $\po$ preserves cardinals.
        \item \label{Clause: splitting forcing lifting measures} For every normal measure $U\in V$ on $\kappa$,
        \begin{enumerate}
            \item If $I\in U$, there are $\tau$ distinct lifts of $U$ to normal measures on $\kappa$ in $V[G]$, $\la U^*_\eta \colon \eta<\tau \ra$, where, for every $\eta<\kappa$, $U^*_\eta$ is generated in $V[G]$ from $U\cup \{S_\eta\}$, where 
        $$S_\eta = \{ \alpha<\kappa \colon \left(\bigcup G\right)(\alpha) = f_\eta(\alpha) \}. $$
            \item If $I\notin U$, $U$ generates a normal measure $U^*\in V[G]$ on $\kappa$ in $V[G]$.
        \end{enumerate}
        \item For every normal measure $W\in V[G]$ on $\kappa$,
        \begin{enumerate}
            \item If $I\in W$, there exists $U\in V$ with $I\in U$, and some $\eta<\tau$, such that $W = U^*_\eta$. 
            \item If $I\notin W$, there exists $U\in V$ such that $I\notin U$ and $W = U^*$.
        \end{enumerate}
    \end{enumerate}
\end{theorem}

\begin{remark}
    A situation of particular interest is when $I$ is the set of all inaccessible cardinals below $\kappa$. In this case, the splitting forcing produces a generic extension in which every normal measure on $\kappa$ has the form $U^*_\eta$ for some normal measure $U\in V$ on $\kappa$ and $\eta<\tau$. In particular, if the ground model carries a unique normal measure on $\kappa$ (for instance, the ground model is $L[U]$), the generic extension has exactly $\tau$ normal measures on $\kappa$. This provides an alternative proof for the Friedman-Magidor Theorem (\cite{FriedmanMagidor}), showing that, consistently from a measurable cardinal, for every $\tau\leq \kappa^{++}$ there exists a measurable cardinal $\kappa$ with $\tau$ normal measures.\footnote{The case $\tau = \kappa^{++}$ follows from the Kunen-Paris Theorem (see \cite{KunenParis}).}
\end{remark}

\begin{lemma}\label{Lemma: How embeddings lift after the splitting forcing}
    Let $j\colon V\to M$ be an elementary embedding satisfying the assumptions of Lemma \ref{Lemma: lifting elementary embeddings with fusion}. Suppose that $\tau\leq \kappa^{+}$, and let $\po = \po^{\tau,I}$. Suppose $G\subseteq \po$ is generic over $V$. Then:
    \begin{enumerate}
    
        \item If $\kappa\in j(I)$,  $j$ lifts to $V[G]$ in exactly $\tau$ ways. That is, 
            \begin{enumerate}
            \item \label{Clause: eta lifts}for every $\eta<\tau$, there exists $H_\eta\subseteq j(\po)$ generic over $M$ and an elementary embedding $j^*_\eta \colon V[G]\to M[H_\eta]$ extending $j$.
            \item \label{Clause: every lift is one of eta many}every elementary embedding $j'\colon V[G]\to M'$ that extends $j$ has the form $j^*_\eta$ for some $\eta<\tau$ (in particular, $M' = M[H_\eta]$). 
            \end{enumerate}        

        \item If $\kappa\notin j(I)$,  $j$  uniquely lifts to $V[G]$. That is, 
            \begin{enumerate}
            \item \label{Clause: unique lift - existence}there exist $H\subseteq j(\po)$ generic over $M$ and an elementary embedding $j^*_\eta \colon V[G]\to M[H]$ extending $j$.
            \item \label{Clause: unique lift - uniqueness} every elementary embedding $j'\colon V[G]\to M'$ that extends $j$ is equal to $j^*$ (in particular, $M'= M[H]$).
            \end{enumerate}        
    
    \end{enumerate}
\end{lemma}

\begin{proof}
We verify all the clauses in turn.

\smallskip

    \eqref{Clause: eta lifts}:  Every $\eta<\tau$  defines a $V[G]$-generic set $g_\eta$ 
    for the poset  $\qo_\kappa = j( \la \qo_\alpha \colon \alpha<\kappa \ra )(\kappa)$. Thus by Lemma \ref{Lemma: lifting elementary embeddings with fusion}, $j$ lifts to $j^*_\tau \colon V[G]\to M[H_\eta]$, where $H_\eta = G*{g_\eta}* (j[G]\setminus \kappa+1).$  

    \smallskip

 \eqref{Clause: every lift is one of eta many}: By  Laver's Ground Model Definability Theorem (Theorem \ref{thm: Laver ground model definability thm}), $H = j'(G)$ is $j(\po)$-generic over $j'(V)$ and $j[G]\subseteq H$.
 
 Here, by $j'(V)$ we mean the class definable in $M'$ using the same formula that defines $V$ inside $V[G]$, when $j'$ is applied to the parameters appearing in the formula. We argue that $j'(V) = M$. Given an ordinal $\alpha\in {\rm{Im}}(j)$, write $\alpha = j(\beta) = j'(\beta)$ for some ordinal $\beta$, and note that $(V_\alpha)^{j'(V)} = j'(V_\beta)$ since $V_\beta$ is the $\beta$-th rank initial segment of the ground $V$ of $V[G]$. Since $j'\supseteq j$, $(V_\alpha)^{j'(V)} = j(V_\beta)$. It follows that $$j'(V) = \bigcup_{\alpha\in \ord} (V_\alpha)^{j'(V)} =  \bigcup_{\beta\in \ord} j(V_\beta) = M.\footnote{Here we use that $j''\ord$ is an unbounded subclass of $\ord$.} $$
    
    Finally, let $\eta = j'(G)(\kappa)$. Then {$\eta<j(f_\tau)(\kappa)=\tau$, and $H_\eta\subseteq H$.} 
    Therefore, by maximality of generic filters,  $H = H_\eta$. Ergo,  $j' = j^*_{\eta}$.

    \smallskip

    \eqref{Clause: unique lift - existence}: Note that $\qo_\kappa = j( \la \qo_\alpha \colon \alpha<\kappa \ra )(\kappa)$ is the trivial forcing in the case when $\kappa\notin j(I)$. In particular, by Lemma \ref{Lemma: lifting elementary embeddings with fusion}, $j[G]$ generates a $j(\po)$-generic $H\subseteq j(\po)$ set over $M$, and $j$ lifts to $j^* \colon V[G]\to M[H]$. 

\smallskip
    
     \eqref{Clause: unique lift - uniqueness}: As in the proof of clause \eqref{Clause: every lift is one of eta many}, $j'(G)$ is $j(\po)$-generic over $j'(V) = M$ and contains $j[G]$. Since $j[G]$ already generates a generic $H\subseteq j(\po)$ over $M$, we deduce that $j'(G) = H$, $M' = M[H]$, and $j' = j^*$.
\end{proof}

We are now in a position to prove Theorem~\ref{thm: properties of the splitting forcing}:

\begin{proof}[Proof of Theorem \ref{thm: properties of the splitting forcing}]
Let us prove each clause in turn: 

\smallskip

(1). This was proved in Lemma \ref{Lemma: nonstatsupport iterations preserving GCH} -- note that the proof that $\kappa$ and $\kappa^+$ remain cardinals does not rely on any GCH assumptions whatsoever.

 \smallskip
 
 (2). Given a normal measure $U\in V$ on $\kappa$, let $\la j^*_\eta \colon \eta<\tau \ra$ be all the lifts of $j_{U_0}$ from Lemma \ref{Lemma: How embeddings lift after the splitting forcing}. For each $\eta<\tau$, let $U^*_\eta$ be the normal measure on $\kappa$ derived from $j^*_\eta$ using $\kappa$ as a seed. Then each $U^*_\eta$ is a normal measure on $\kappa$ extending $U\cup \{ S_\eta \}$, where $$S_\eta = \{ \alpha<\kappa \colon G(\alpha) = f_\eta(\alpha) \}.$$
    In fact, $U\cup \{ S_\eta \}$ generates (in $V[G]$) $U^*_\eta$ for every $\eta<\tau$. Indeed, given $X\in U^*_\eta$, let  $\dot{X}$ be a $\po$-name for it. Since $X\in U^*_\eta$, there exists $q\in H_\eta$ (where $H_\eta$ is as in the notation of Lemma \ref{Lemma: How embeddings lift after the splitting forcing}) such that $q\Vdash \check{\kappa}\in j_{U_0}(\dot{X})$. However, since $H_\eta = G*g_\eta*\left( j_{U_0}[G]\setminus \kappa+1\right)$, there exists $p\in G$ such that 
$$p^{\frown} \{ \langle \kappa, \eta \rangle \}^{\frown} (j_{U_0}(p)\setminus \kappa+1)\Vdash \check{\kappa}\in j_{U_0}(\dot{X}).$$ In particular,
    $$B = \{ \alpha<\kappa \colon {p\restriction \alpha} ^{\frown} \{ \langle \alpha, f_{\eta}(\alpha) \rangle \}^{\frown} (p\setminus \alpha+1) \Vdash \check{\alpha}\in \dot{X} \}\in U.$$
    Thus, in $V[G]$, $B\cap S_\eta \subseteq X$. 

    \smallskip
    
    (3). Let us argue now that every normal measure $W\in V[G]$ has the form $U^*_\eta$ for some normal $U\in V$  and $\eta<\tau$. Given such $W$, denote by $j_W\colon V[G]\to M^*$ its ultrapower embedding. By Hamkins' Gap Forcing Theorem, $j_W\restriction V$ is definable in $V$, since $\po$ has a gap below $\kappa$. Denote $U = W\cap V$. Then $U\in V$ since, in $V$, $U = \{ X\subseteq \kappa \colon \kappa\in (j_W\restriction V)(X) \} $. Next, denote $\eta = j_W(G)(\kappa)$. Then, by the normality of $W$, $\eta< \tau$ 
    and $S_\eta\in W$. It follows that $U\cup \{ S_\eta \}\subseteq W$ and, since $U^*_\eta$ is generated from $U\cup \{ S_\eta \}$, $W = U^*_\eta$.
    \end{proof}

    Finally, we argue that the splitting forcing preserves the Mitchell order.

    \begin{lemma}[Preservation of the Mitchell order]\label{lemma: the splitting forcing preserves Mitchell order}
        Let $I\subseteq \kappa$, $\tau<\kappa^+$ and $\po = \po^{\tau, I}$ be as in Theorem \ref{thm: properties of the splitting forcing}.   Let $G\subseteq \po$ be generic over $V$.
        \begin{enumerate}
            \item \label{Clause: Mitchell order 1}Assume that $U_0, U_1$ are normal measures on $\kappa$ in $V$. Let $W_0, W_1\in V[G]$ be normal measures on $\kappa$ with  $U_0\subseteq W_0$ and $U_1\subseteq W_1$. Then,
            $$ V\vDash U_0 \vartriangleleft U_1 \iff  V[G]\vDash W_0 \vartriangleleft W_1. $$

            \item \label{Clause: Mitchell order 2}For normal measures on $\kappa$, $U\in V$ and $W\in V[G]$, such that $U\subseteq W$,
            $$ \left(o(W)\right)^{V[G]} =\left( o(U)\right)^V.\footnote{Note that this is an equation between ordinals, and it remains true regardless of cardinals being collapsed in $V[G]$.} $$
        
        \end{enumerate}
        
    \end{lemma}

    \begin{proof}
        Clause \eqref{Clause: Mitchell order 2} is an immediate corollary of clause \eqref{Clause: Mitchell order 1} and the facts that every normal measure on $\kappa$ in $V$ lifts to a normal measure on $\kappa$ in $V[G]$, and every normal measure on $\kappa$ in $V[G]$ is a lift of a normal measure on $\kappa$ in $V$. Thus, we concentrate in proving  clause \eqref{Clause: Mitchell order 1}. Let $U_0, U_1, W_0, W_1$ be as above. We first argue that
        $$ V\vDash U_0 \vartriangleleft U_1 \Longrightarrow  V[G]\vDash W_0 \vartriangleleft W_1. $$
        
        Assume first that $I\in U_0$, namely $W_0 = (U_0)^*_\eta$ for some $\eta<\tau$.  Since $G$ is $\po$-generic over $M_{U_1}$ and $U_0\in M_{U_1}$, $U_0 \cup \{ S_\eta \}$ generates a normal measure $\left( (U_0)^*_\eta \right)^{M_{U_1}[G]}$ in $M_{U_1}[G]$. It suffices to argue that $(U_0)^*_\eta = \left( (U_0)^*_\eta \right)^{M_{U_1}[G]}$. For that, it would be enough to prove that $$\left( \mathcal{P}(\kappa) \right)^{V[G]} = \left( \mathcal{P}(\kappa) \right)^{M_{U_1}[G]}$$
        because the base of both measures is the same.

        Assume that $X\in V[G]$ is a subset of $\kappa$, and fix a $\po$-name $\dot{X}\in V$ for it. By the Fusion Lemma \ref{Lemma: fusion lemma}, we can find $p\in G$ and a club $C\subseteq \kappa$ such that, for every $\alpha\in C$,
        $$p\restriction (\alpha+1) \Vdash \exists X_\alpha\subseteq \alpha \left( p\setminus (\alpha+1) \Vdash \dot{X}\cap \alpha = X_\alpha \right).$$
        In particular, for every $\alpha\in C$, there exists a witnessing $\po_{\alpha+1}$-name $\dot{X}_\alpha\in V_\kappa$. Since $C, \la \dot{X}_\alpha \colon \alpha\in C \ra\in M_{U_1}$ and $p\in G$, we can reconstruct $X$ in $M_{U_1}[G]$, as $X = \bigcup_{\alpha\in C} \left( \dot{X}_\alpha\right)_{G_\alpha}\in M_{U_1}[G]$.

        Assume now that $I\notin U_0$. In this case, $W_0$ is generated by $U_0$ in $V[G]$. Since $U_0\in M_{U_1}$, it suffices to prove that $(U_0)^* = \left( (U_0)^* \right)^{M_{U_1}[G]}$, which again follows from the fact that $\left( \mathcal{P}(\kappa) \right)^{V[G]} = \left( \mathcal{P}(\kappa) \right)^{M_{U_1}[G]}$.

        We proceed now and prove that for every $U_0, U_1, W_0, W_1$ as in the formulation of the lemma, 
        $$  V[G]\vDash W_0 \vartriangleleft W_1 \Longrightarrow V\vDash U_0 \vartriangleleft U_1. $$

       By the analysis in the proof of Theorem \ref{thm: properties of the splitting forcing}, $\text{Ult}(V[G], W_1)$ has the form $M_{U_1}[H]$ for $H\subseteq j_{U_1}(\po)$-generic\footnote{{Indeed, we showed that either  $W_1=(U_1)^*_\eta$ (if $I\in U_1$, for some $\eta<\tau$), or $W_1 = (U_1)^*$ (if $I\notin U_1$). In each case, the ultrapower embedding associated with $W$ is an embedding $j^*\colon V[G]\rightarrow M_{U_1}[H]$ for some $H\subseteq j_{U_1}(\po)$, as proved in Theorem \ref{thm: properties of the splitting forcing}.}} over $M_{U_1}$ with $H\cap \po = G$. 
       Since $j_{U_1}(\po)\setminus \kappa$ is sufficiently closed, $W_0 \in M_{U_1}[G]$. Finally,  $\po$ has a gap below $\kappa$, ergo $U_0 = W_0\cap M_{U_1}\in M_{U_1}$, as desired.
    \end{proof}

    \begin{corollary}
        Let $\lambda$ be a supercompact cardinal, and denote by $\kappa<\lambda$ the least measurable cardinal. Assume that $\kappa$ has a unique normal measure.\footnote{This follows, for example, from the linearity of the Mitchell order.} Fix $\tau\leq\kappa^{+}$. Then, in a forcing extension, $\lambda$ is supercompact, $\kappa$ is a measurable cardinal and it carries exactly $\tau$ normal measures.
    \end{corollary}

    \begin{proof}
        Let $I$ be the set of all inaccessible cardinals below $\kappa$, and let $\po = \po^{\tau,I}$ be the splitting forcing (defined as an iterated forcing of length $\kappa$). Since $\po$ is small relative to $\lambda$, $\lambda$ remains  supercompact  in $V[G]$, for $G\subseteq \po$ generic over $V$. By Theorem \ref{thm: properties of the splitting forcing}, $\kappa$ carries exactly $\tau$ normal measures in $V[G]$.
    \end{proof}
For later use (see e.g., Theorem~\ref{thm: IdentityCrises + Normal measures})  we prove the following technical lemma showing that the splitting forcing $\mathbb{P}^{\tau,I}$ preserves supercompact cardinals modulo a suitable choice of $I$:
\begin{lemma}\label{lemma: the spaced splitting forcing preserves supercompactness}
        Assume {\rm{GCH}}. Let $\kappa$ be a supercompact cardinal and fix $\kappa_0 < \kappa$. Let $I$ be the set of inaccessible cardinals in $(\kappa_{0},\kappa)$ that are limit of strong cardinals. Fix $\tau\leq \kappa^+$. Denote $\po = \mathbb{P}^{\tau,I}$. Then 
        \begin{enumerate}
            \item \label{Clause: spaced splitting forcing 1} $\po$ is $(\kappa_0)^+$-directed closed, and $\po$ preserves cardinals and the {\rm{GCH}}.
            \item \label{Clause: spaced splitting forcing 2} $\po$ preserves the supercompactness of $\kappa$.
            \item \label{Clause: spaced splitting forcing 3} Let $G\subseteq \po$ be generic over $V$. Then every normal measure $U\in V$ on $\kappa$ lifts in $\tau$ ways to measures $\la U^*_\eta \colon \eta<\tau \ra$ on $\kappa$ in $V[G]$. Furthermore, every normal measure $W\in V[G]$ on $\kappa$ is such a lift.
        \end{enumerate}
    \end{lemma}

    \begin{proof}
        Clearly, $\po$ is $(\kappa_0)^+$-directed closed. The fact that $\po$ preserves cardinals and GCH follows from Lemma \ref{Lemma: nonstatsupport iterations preserving GCH}. Let us proceed to clause \eqref{Clause: spaced splitting forcing 2}. 
        
        Let $G\s \mathbb{P}$ a $V$-generic. 
        For any singular strong limit cardinal $\lambda > \kappa$ with $\cf(\lambda)>\kappa$ we show that $\kappa$ is $\lambda$-supercompact in $V[G]$. Indeed, let $j \colon V \to M$ be the ultrapower embedding derived from a fine, normal measure on $\mathcal{P}_\kappa(\lambda)$ of Mitchell order $0$; in particular, $\kappa$ is not $\lambda$-supercompact in $M$.

        Because $\lambda$ is a strong limit cardinal, it follows that $\kappa$ is $<\lambda$-supercompact in $M$. In particular, there are no strong cardinals in the interval $(\kappa,\lambda)$ in $M$, for if there were a strong cardinal $\mu$ with $\kappa < \mu < \lambda$ in $M$ then $\kappa$ would be $\mu$-supercompact in $M$,  hence fully supercompact therein, thus contradicting our choice of the measure $U$. 

        This observation ensures that the forcing $j(\po)\restriction (\kappa, \lambda]$ is trivial in $M$. 
        
        Let us lift $j$ by constructing a generic for $j(\po)$ over $M$. 
        Note that $$j(\po)\restriction (\kappa+1) = \po \ast \dot{\qo}_\kappa,$$ where $\dot{\qo}_\kappa$ is forced to be an atomic forcing that picks an ordinal $\eta<\tau$. Also $j(\po)\setminus (\kappa+1)$ is $\lambda^{+}$-directed closed since there are no inaccessible cardinals that are limit of strong cardinals in the interval $(\kappa, \lambda]$.

        Fix some $\eta<\tau$ (say, $\eta = 0$) 
        and let $G*g_\eta$ be the corresponding $j(\po)\restriction (\kappa+1)$-generic set over $M$. Since $M$ is closed under $\lambda$ sequence in $V$ and $|\po|< \lambda$, $M[G] = M[G\ast g_\eta]$ is closed under $\lambda$-sequences in $V[G]$. Note that $j[G]\in M[G]$ since $j\restriction \po \in M$. Let 
        $$s = \bigcup j[G]\setminus (\kappa+1) = \bigcup\{ j(p)\setminus (\kappa+1) \colon p\in G \}. $$
        Then $s\in j(\po)\setminus (\kappa+1)$ is a master condition. Construct, in $V[G]$, a generic $H\subseteq j(\po)\setminus (\kappa+1)$ with $s\in H$. This construction is possible because $V[G]$ has a list of order type $|j(\kappa)^{++, M}| = \lambda^{+}$ of all dense subsets of $j(\po)\setminus (\kappa+1)$,\footnote{here we used GCH and the fact that $\cf(\lambda)>\kappa$.} and $j(\po)\setminus (\kappa+1)$ is $\lambda^+$-closed in $V[G]$. It follows that $j\colon V\to M$ lifts to $j^*\colon V[G]\to M[G\ast g_\tau \ast H]$, and that  $M[G\ast g_\tau \ast H]$ is closed under $\lambda$-sequences inside $V[G]$. Therefore $j^*$ witnesses that $\kappa$ is $\lambda$-supercompact in $V[G]$. Since $\lambda$ can be chosen arbitrarily high, $\kappa$ is fully supercompact in $V[G]$.

        \smallskip

        Clause \eqref{Clause: spaced splitting forcing 3} follows from clause \eqref{Clause: splitting forcing lifting measures} of Theorem \ref{thm: properties of the splitting forcing}, and the following facts:
        \begin{enumerate}
            \item Every normal measure $U\in V$ on $\kappa$ concentrates on $I$ (this follows from $\kappa$ being supercompact in $V$, and in particular, an inaccessible limit of strong cardinals). 
            \item Every normal measure $W\in V[G]$ concentrates on $I$ as well (this follows from $\kappa$ being supercompact in $V[G]$ and the fact that $\po$ has a gap below $\kappa$, so every large enough strong cardinal below $\kappa$ in $V[G]$ was already a strong cardinal in $V$. See also (\cite[Corollary 13]{HamkinsGap}).
        \end{enumerate}
        This completes the proof of the lemma.
    \end{proof}

    Finally, we would like to adjust Lemma \ref{lemma: the spaced splitting forcing preserves supercompactness} to the case where $\tau = \kappa^{++}$.
    
    \begin{lemma}\label{lemma: kappa^++ many normal measures}
        Assume GCH. Let $\kappa$ be a supercompact and fix $\kappa_0<\kappa$. Let $I$ be the set of inaccessible cardinals in $(\kappa_0, \kappa)$ that are limit of strong cardinals. Let $\tau = \kappa^{++}$. Let $\po = \po^{\tau,I}$ be the $I$-spaced nonstationary support iteration\footnote{An Easton support may be used here as well.} $\la \po_\alpha, \dot{\qo}_\alpha \colon \alpha\leq\kappa \ra$, where for every $\alpha\in I\cup \{\kappa\}$, $\dot{\qo}_\alpha = (Add(\alpha,1))^{V^{\po_\alpha}}$ (note that a forcing is being done on $\kappa$). Then:
        \begin{enumerate}
            \item $\po$ is $(\kappa_0)^+$-directed closed, and $\po$ preserves cardinals and the {\rm{GCH}}.
            \item $\po$ preserves the supercompactness of $\kappa$.
            \item Let $G\subseteq \po$ be generic over $V$. Every $U\in V$ lifts in $\kappa^{++}$ ways to a normal measure on $\kappa$ in $V[G]$. Furthermore, every normal measure $W\in V[G]$ is such a lift.
        \end{enumerate}
    \end{lemma}

    \begin{proof}(Sketch)
        The facts that $\po$ is $(\kappa_0)^+$-directed closed, preserves cardinals and preserves the GCH follow as in Lemma \ref{lemma: the spaced splitting forcing preserves supercompactness}. The proof that $\po$ preserves the supercompactness of $\kappa$ is also very similar to the proof from Lemma \ref{lemma: the spaced splitting forcing preserves supercompactness}: For $\lambda>\kappa$ a strong limit cardinal, lift a $\lambda$-supercompactness embedding $j\colon V\to M$ first to  
        $j^*\colon V[G_{<\kappa}]\to M[G*H]$, where $G = G_{<\kappa}*G(\kappa)$ is $\po$-generic over $V$, and $H\subseteq j(\po)\setminus (\kappa+1)$ is generic over $M[G]$ that contains a suitable master condition $s$ as in Lemma \ref{lemma: the spaced splitting forcing preserves supercompactness}. The main difference is that now we need to construct $h\in V[G]$ which is $j(\po)(j(\kappa))$-generic over $M[G*H]$ with $j^*[G(\kappa)]\subseteq h$. For that, we regard $\bigcup g\in M[G*H]$ as a master condition, and pick any $h\in V[G]$ which is $(Add(j(\kappa),1))^{M[G*H]}$-generic over $M[G*H]$ with $h\restriction \kappa = g$. This ensures that $j^*$ lifts to $j^{**}\colon V[G]\to M[G*H*h]$, witnessing $\lambda$-supercompactness of $\kappa$ in $V[G]$.

        Finally, proceed to the classification of normal measures on $\kappa$ in $V[G]$. Since $\po$ has a gap below $\kappa$, every normal measure $W\in V[G]$ on $\kappa$ extends a normal measure $U\in V$ on $\kappa$. The fact that $U$ lifts in $\kappa^{++}$-many ways follows by lifting $j_U\colon V\to M_U$ to $j^* \colon V[G]\to M[G*H*h]$, where:
        \begin{itemize}
            \item $H = j_U[G_{<\kappa}]\setminus (\kappa+1)$.
            \item $h\subseteq Add(j(\kappa), 1)$ is generic over $M[G*H]$ with $h\restriction \kappa = g$. 
        \end{itemize}
        Standard arguments (see  \cite[Proposition 8.1]{CummingsHandbook}) show that there are $2^{\kappa^+} = \kappa^{++}$ many such generics $h$ over $M[G*H]$. Each $h$ induces a lift $j^*_h \colon V[G]\to M_U[G*H*h]$ of $M_U$. The normal measures $\la U^*_h \colon h \text{ as above} \ra$ are pairwise distinct lifts of $U$, as each gives rise to a different ultrapower embedding $j^*_h$. Thus $U$ lifts in $\kappa^{++}$-many ways.\footnote{We remark that there are other lifts, since $G(\kappa)$ doesn't have to be taken as the $Add(\kappa, 1)^{M[G]}$-generic over $M[G]$.}
    \end{proof}
    
    \subsection{Laver's indestructibility  via nonstationary support iterations}

    Next, we prove a version of Laver's theorem on making a supercompact cardinal indestructible under a certain class of forcings (\cite{Lav}). The wrinkle  here is that the  preparatory forcing is an $I$-spaced nonstationary-support iteration instead of an Easton support iteration. This allows for a much finer control of certain normal measures on $\kappa$ in the resulting generic extension.

    Assume that $\kappa$ is a supercompact cardinal. Let $\ell\colon \kappa\to V_\kappa$ be a Laver function for $\kappa$. Recall that this means that, for every $\lambda\geq \kappa$ and $x\in H_{\lambda^+}$, there exists a fine, normal ultrafilter $\mathcal{U}$  on $\mathcal{P}_{\kappa}(\lambda)$ such that $j_{\mathcal{U}}(\ell)(\kappa) = x$.

    \begin{theorem}[Laver indestructibility]\label{theorem: Laver indestructibility}
        Assume $\mathrm{GCH}$, and let $\kappa$ be a supercompact cardinal. There exists poset $\po\in V_{\kappa+2}$ such that whenever $G\subseteq \po$ is generic over $V$:
        \begin{enumerate}
            \item $\kappa$ is supercompact in $V[G]$, and its supercompactness is indestructible under $\kappa^+$-directed closed forcings.\footnote{The standard version of the indestructibility  theorem, in which the indestructibility is under $\kappa$-directed closed forcing, could also be proved, but it does not fulfill clause \eqref{clause: indestructibility 2} above.}
            \item\label{clause: indestructibility 2} Every normal measure $U\in V$ on $\kappa$ of Mitchell order $0$ generates a normal measure $U^*\in V[G]$ on $\kappa$ of Mitchell order 0. Furthermore, every normal measure $W\in V[G]$ on $\kappa$ of Mitchell order $0$ has the form $U^*$ for some normal measure $U\in V$ on $\kappa$ of Mitchell order $0$.
        \end{enumerate}
    \end{theorem}

    \begin{proof}
    Our proof follows \cite[Theorem 24.12]{CummingsHandbook}. 
    
    Let $I\s \kappa$ be the stationary set given by    $$\{ \alpha<\kappa \colon \alpha \text{ is measurable and } \forall \beta<\alpha \left( \exists \gamma<\kappa\, \exists x\in V_\gamma \ \left( \ell(\beta) = (\gamma,x) \right) \to \gamma< \alpha \right) \}$$
    We define an $I$-spaced nonstationary support iteration $$\la \po_\alpha,  \dot{\qo}_\alpha\colon \alpha<\kappa \ra$$ such that for every $\alpha<\kappa$:
    \begin{itemize}
        \item If $\alpha$ is inaccessible, $\po_\alpha$ is the nonstationary support limit of the iteration $\la \po_\beta \colon \beta < \alpha \ra$. Else, $\po_\beta$ is the inverse limit of $\la \po_\beta \colon \beta<\alpha \ra$.
        \item $\dot{\qo}_\alpha$ is forced to be trivial, unless $\alpha\in I$ and $\ell(\alpha)$ is a pair $\left( \gamma, x \right)$ such that:
        \begin{itemize}
            \item $\gamma < \min(I\setminus (\alpha+1))$.
            \item $x\in V_\gamma$ is a $\po_\alpha$-name for an $\alpha^+$-directed closed poset.
        \end{itemize}
        In that case, $\po_\alpha$ forces that $\dot{\qo}_\alpha$ is the poset $x$.
    \end{itemize}
    Let $\po$ be the nonstationary support limit of the forcings $\la \po_\alpha \colon \alpha<\kappa \ra$. 

    \smallskip

    (1) Let $G\subseteq \po$ be generic over $V$.  Let $\qo\in V[G]$ be a $\kappa^+$-directed closed poset, and $g\subseteq \qo$ generic over $V[G]$. We argue that $\kappa$ is supercompact in $V[G\ast g]$ (in particular, $\kappa$ is supercompact in $V[G]$ by taking $\qo=\{\one \}$). 

    Assume that $\lambda > \kappa$ bounds the rank of $\dot{\qo}$ in $V$. We argue that $\kappa$ is $\lambda$-supercompact in $V[G*g]$. Let $\mu = 2^{2^\lambda}$, and let  $\mathcal{U}$ be a supercompactness measure on $\mathcal{P}_{\kappa}(\mu)$ with $j_{\mathcal{U}}(\ell)(\kappa) = \left( \mu, \dot{\qo}\right)$. In particular, in $M_{\mathcal{U}}$, $\dot{\qo} = j_{U_0}(\la \qo_\alpha\colon \alpha<\kappa \ra)(\kappa)$ and $G*g$ is $j_{\mathcal{U}}(\po)\restriction (\kappa+1)$-generic over $M_{\mathcal{U}}$. 
    Furthermore, $\min(j_{\mathcal{U}}(I) \setminus (\kappa+1) )>\mu$, so the forcing $j_{\mathcal{U}}(\po)\setminus (\kappa+1)$ in $M_{\mathcal{U}}[G*g]$ is $\mu^{+}$-directed-closed. Since $\po*\dot{\qo}$ is $\mu$-c.c., $V[G*g]\vDash {}^\mu M_{\mathcal{U}}[G*g] \subseteq M_{\mathcal{U}}[G*g] $. Therefore,  $j_{\mathcal{U}}(\po)\setminus (\kappa+1)$ is $\mu^+$-directed-closed in $V[G*g]$ as well. Since the set $j_{\mathcal{U}}[G]$ belongs to $M_{\mathcal{U}}[G]$, we can find a master condition $q^*\in j_{\mathcal{U}}(\po)$ such that for every $p\in G$, $q\leq j_{\mathcal{U}}(p)$ and $\kappa\notin \supp(q)$. Let $H\subseteq j_{\mathcal{U}}(\po)\setminus (\kappa+1)$ be generic over $V[G*g]$ with $q^*\setminus(\kappa+1)\in H$. Since $j_{\mathcal{U}}[G]\subseteq G*g*H$, we can lift the embedding $j_{\mathcal{U}}\colon V\to M_{\mathcal{U}}$ to $j'\colon V[G]\to M_{\mathcal{U}}[G*g*H]$, where $j'$ is definable in $V[G*g]$. Finally, since $j'(\qo)$ is $j_{\mathcal{U}}(\kappa^+)$-directed closed and $j'[g]$ has size at most $\mu\leq j_{\mathcal{U}}(\kappa^+)$ and belongs to $M_{\mathcal{U}}[G*g*H]$, we can find a master condition $r\in j'(\qo)$, such that $r\leq j'(s)$ for every $s\in g$.  Let $h\in V[G*g*H]$ be $j'(\qo)$-generic over $M_{\mathcal{U}}[G*g*H]$. In $V[G*g*H*h]$, lift $j'$ to $j^*\colon V[G]\to M_{\mathcal{U}}[G*g*H*h]$; this is possible since $j'[g]\subseteq h$. 
    
    Let $\mathcal{W}\in V[G*g*H*h]$ be the fine, normal measure on $\left(\mathcal{P}_{\kappa}(\lambda)\right)^{V[G*g*H*h]}$ derived from $j^*$ using $j^*[\lambda]$ as a seed. Since $j_{\mathcal{U}}(\kappa^+) > \mu$, the forcing $j'(\qo)$ does not add new subsets to $2^{2^\lambda}$. This, combined with the fact that $j_{\mathcal{U}}(\po)\setminus (\kappa+1)$ is $\mu^+$-directed-closed in $V[G*g]$, implies that $\mathcal{W}$ already belongs to $V[G*g]$, so there exists a fine, normal measure on $\left(\mathcal{P}_{\kappa}(\lambda)\right)^{V[G*g]}$ in $V[G*g]$. Since $\lambda$ was arbitrary large, $\kappa$ remains supercompact in $V[G*g]$.

    \smallskip

   (2)  We prove this via two claims:
   \begin{claim}\label{Claim: indestructibility proof, claim 1}
     Every normal measure $U\in V$ on $\kappa$ of Mitchell order $0$ generates a normal measure on $\kappa$ in $V[G]$.  
   \end{claim}
   \begin{proof}[Proof of claim]
      Let $U\in V$ be a normal measure of order $0$ on $\kappa$, and $j_{U_0}\colon V\to M_U$ be its ultrapower embedding. Note that $\kappa\notin j_{U_0}(I)$ since $I$ consists of measurable cardinals and $U$ has Mitchell order $0$. Therefore $j_{U_0}(\la \dot{\qo}_\alpha \colon \alpha<\kappa \ra)(\kappa)$ is forced by $j_{U_0}(\po)_ \kappa = \po$ to be the trivial forcing.
      By Lemma \ref{Lemma: lifting elementary embeddings with fusion}, $j_{U_0}$ lifts in $V[G]$ to an elementary embedding 
    $$j^*\colon V[G]\to M_{U}[j_{U_0}[G]]$$ witnessing the measurability of $\kappa$ in $V[G]$. 
    
    Let $U^*$ be the normal measure derived from $j^*$
using $\kappa$ as a seed.\footnote{It's not hard to prove that this extended embedding is the ultrapower embedding $j^{V[G]}_{U^*}$, where $U^*\in V[G]$ is the normal measure derived from it using $\kappa$ as a seed.} Namely,
 $$ U^* = \left\{ X\in (\mathcal{P}(\kappa))^{V[G][H]} \colon  \exists p\in G \left( j_{U_0}(p)\Vdash \check{\kappa}\in j_{U_0}(\dot{X}) \right)  \right\}.$$
 Given $X\in U^*$, there exists a $\po$-name $\dot{X}$ for $X$ and a condition $p\in G$ such that $j_{U}(p)\Vdash \check{\kappa}\in j_{U}(\dot{X})$. In particular, the set $A = \{ \alpha<\kappa \colon p\Vdash \check{\alpha}\in \dot{X} \}$ belongs to $U$, and since $p\in G$, $A\subseteq X$. Therefore, $U^*$ is generated in $V[G]$ by $U$. 
   \end{proof}

\begin{claim}\label{claim: Mitchell order 0 generated by measures of order 0}
    Every normal measure $W\in V[G]$ on $\kappa$ of Mitchell order $0$ is generated by a normal measure $U\in V$ on $\kappa$ of Mitchell order $0$.
\end{claim}
\begin{proof}[Proof of claim]
    Assume that $W\in V[G]$ is a normal measure on $\kappa$ in $V[G]$ of Mitchell order $0$. Denote the ultrapower embedding associated to $W$ by $$j^{V[G]}_W\colon V[G]\to M[H].$$ This is  an elementary embedding, where $H = j_W(G)$ is $j_W(\po)$-generic over some ground model $M$. 
    
    Since $\po$ has a gap below $\kappa$,  Hamkins' Gap Forcing Theorem (Theorem \ref{thm: Gap forcing theorem}) ensures that  $U = W\cap V$ belongs to $V$ and $M\subseteq V$. It suffices to argue that $U$ has Mitchell order $0$ in $V$; this will conclude the proof, since then $W$ must be equal to the measure $U^*$ generated by $U$ in $V[G]$. Indeed, assume that $U$ does not have Mitchell order $0$ in $V$. {Then $U$ concentrates on $I$.} 
    In particular, $I\in W$. We will produce a contradiction by showing that every cardinal $\beta\in I$ remains measurable in $V[G]$, contradicting the fact that $W$ has Mitchell order $0$ in $V[G]$. Indeed, fix $\beta\in I$. By repeating the argument from Claim \ref{Claim: indestructibility proof, claim 1}, every normal measure $U_\beta\in V$ on $\beta$ of Mitchell order $0$ generates a normal measure on $\beta$ in $V[G_{\beta}]$. Since $\dot{\qo}_\beta$ is forced to be $\beta^+$-directed closed, the forcing $\po\setminus \beta$ does not add new subsets to $\beta$, so $\beta$ remains measurable in $V[G]$, as desired. 
\end{proof}
   This concludes the proof of Theorem~\ref{theorem: Laver indestructibility}.
    \end{proof}

A similar argument yields a version of the indestructibility theorem under forcings which preserve GCH.

\begin{theorem}[Laver indestructibility with GCH and cardinal preservation]\label{theorem: Laver indestructibility+GCH}
        Assume $\mathrm{GCH}$, and let $\kappa$ be a supercompact cardinal. There exists poset $\po\in V_{\kappa+2}$ such that whenever $G\subseteq \po$ is generic over $V$:
        \begin{enumerate}
            \item $\kappa$ is supercompact in $V[G]$, and its supercompactness is indestructible under $\kappa^+$-directed closed forcings which preserve cardinals and preserve  \rm{GCH}.
            \item Every normal measure $U\in V$ on $\kappa$ of Mitchell order $0$ generates a normal measure $U^*\in V[G]$ on $\kappa$ of Mitchell order 0. Furthermore, every normal measure $W\in V[G]$ in $\kappa$ of Mitchell order $0$ has the form $U^*$ for some normal measure $U\in V$ on $\kappa$ of Mitchell order $0$.
            \item \rm{GCH} holds in $V[G]$, and every cardinal of $V$ remains a cardinal in $V[G]$.
        \end{enumerate}
    \end{theorem}

    \begin{proof}
        Let $\po$ be a poset similar to the one in the proof of Theorem \ref{theorem: Laver indestructibility}, only requiring that whenever $\alpha\in I$ and $\ell(\alpha)$ is a pair $(\gamma,x)$, $\dot{\qo}_\alpha = x$, provided that:
        \begin{itemize}
            \item  $\gamma< \min(I\setminus (\alpha+1))$.
            \item $x\in V_\gamma$ a $\po_\alpha$-name for an $\alpha^+$-directed closed poset which preserves cardinals and GCH.
        \end{itemize}
        Else, $\dot{\qo}_\alpha$ is trivial. Since $\po$ is now a nonstationary support iteration of cardinal preserving and GCH preserving posets, GCH holds in $V$ in addition to the conclusions of Theorem \ref{theorem: Laver indestructibility}.
        \end{proof}

     \begin{notation}\label{notation: Laver}
        Fix the following notations for variations of the indestructibility preparations introduced in this section:
     \begin{enumerate}
        \item Let $\kappa$ be a supercompact cardinal, and let $\rho<\kappa$ be a regular cardinal. $\mathbb{L}(\kappa,\rho)$ is a $\rho$-directed closed version of the nonstationary support preparation that makes the supercompactness of $\kappa$ indestructible under $\kappa^+$-directed-closed forcings preserving cardinals and the GCH.
        \item Let $\kappa$ be a strong cardinal, and let $\rho<\kappa$ be a regular cardinal. $\mathbb{GS}(\kappa,\rho)$ is a $\rho$-directed closed version of the poset that makes the stronngness of $\kappa$ indestructible under $\kappa^+$-directed-closed forcings. The existence of such a poset was proved by Gitik and Shelah (see \cite{GitikShelah}).\footnote{In fact, the Gitik–Shelah indestructibility preparation applies to Prikry-type forcings with a sufficiently closed direct extension order; see \cite{GitikShelah} for more details. Let us also remark that a nonstationary support variation of the Gitik–Shelah indestructibility preparation exists, but will not be needed in the current paper.} 
     \end{enumerate}
     Such forcings can easily be shown to exist by modifying the proofs of Theorem \ref{theorem: Laver indestructibility+GCH} and the Gitik-Shelah indestructibility preparation from \cite{GitikShelah} so that $\dot{\qo}_\alpha$ is forced to be trivial for every $\alpha<\rho$.
   
    \end{notation}
    
\subsection{Forcing nonreflecting stationary sets}\label{Section: Forcing NR}

    For every Mahlo cardinal $\alpha$, we denote by $\mathbb{NR}(\alpha)$ the forcing notion that adds a nonreflecting stationary set to $\alpha$. We use a modification of the standard forcing in which, for every regular cardinal $\gamma<\alpha$, there exists a dense subset $\mathbb{NR}_\gamma(\alpha)\subseteq \mathbb{NR}(\alpha)$ which is $\gamma$-directed closed. The forcings are defined as follows:
    \begin{itemize}
        \item $\mathbb{NR}(\alpha)$ consists of conditions which are functions $q\colon \beta\to 2$, where $0<\beta<\alpha$, such that:
        \begin{itemize}
            \item For every cardinal  $\lambda\leq \beta$ with uncountable cofinality, there exists a club $C\subseteq \lambda$ such that $q\restriction C$ is identically $0$. 
            \item {For every inaccessible cardinal $\lambda<\beta$, the set $\{\alpha<\beta: \alpha\geq \lambda\,\wedge\, q(\alpha)=1\}$ consists of ordinals of cofinality at least $\lambda$}.
        \end{itemize}
        
        Given $p,q \in \mathbb{NR}(\alpha)$, $q$ extends $p$ if $q$ end-extends $p$.
        \item Given a regular cardinal $\gamma<\alpha$,  $\mathbb{NR}_\gamma(\alpha)$ consists of conditions $q\in \mathbb{NR}(\alpha)$ such that $\gamma<\dom(q)$.
    \end{itemize}
    
    A generic $G\subseteq \mathbb{NR}(\alpha)$ over the ground model $V$ induces a nonreflecting stationary subset of $\alpha$ (see Lemma \ref{lemma: NR adds a nonreflecting stationary set} below).

   \begin{lemma} \label{Claim: NR is uniformly strategically closed}
       Let $\alpha$ be a Mahlo cardinal, and let  $\gamma<\alpha$ be a regular cardinal. Then $\mathbb{NR}(\alpha)$ and $\mathbb{NR}_{\gamma}(\alpha)$ are both $\alpha$-strategically closed.
   \end{lemma} 

   \begin{proof}
    We describe a winning strategy for Player II in the game $G_\alpha(\mathbb{NR}(\alpha))$ (a similar argument can be given for $\mathbb{NR}_\gamma(\alpha)$). Assume that $i<\alpha$ is an even stage and the players have constructed so far a descending sequence of conditions $\la p_j \colon j<i \ra\subseteq \mathbb{NR}(\alpha)$. Assume by induction the following properties:
    \begin{itemize}
        \item For every even $j<i$, $\dom(p_j)$ is an ordinal of the form $\beta_j+1$, and $p_j(\beta_j) =0$.
        \item $\la \beta_j \colon j<i \text{ is even} \ra$ is an increasing, continuous sequence.
        \end{itemize}
    If $i$ is a limit ordinal, let $\beta_i = \sup_{j<i} \beta_j$. Let $p^*_i = \bigcup_{j<i}p_j$. The move of Player II in stage $i$ will then be the condition $p_i = p^*_i \cup \{ \la \beta_i, 0 \ra \}$. 
           
    If $i = i'+1$ is a successor ordinal and $p_{i'}$ is the condition played by Player I in the most recent stage, Player II chooses at stage $i$ an ordinal $\beta_{i}$ strictly above $\dom(p_{i'})$, and plays with the condition $p_{i}\colon \beta_{i}+1\to 2$ such that $p_{i}\restriction \dom(p_{i'}) = p_{i'}$ and $p_{i}\restriction \left((\beta_{i}+1) \setminus \dom(p_{i'})\right)$ is identically $0$.
   \end{proof}

     \begin{lemma}\label{lemma: NR adds a nonreflecting stationary set}
        Assume that $\alpha$ is Mahlo. Let $G\subseteq \mathbb{NR}(\alpha)$ be generic over $V$. Let $S_G = \{ \beta<\alpha \colon \exists p\in G \ (\beta\in \dom(p) \land  p(\beta)= 1) \}$. Then $S_G$ is a nonreflecting stationary subset of $\alpha$.
    \end{lemma}

    \begin{proof}
        By the definition of $\mathbb{NR}(\alpha)$, $S_G$ is nonreflecting; in fact, for every $\lambda<\alpha$ of uncountable cofinality, there exists a ground-model club in $\lambda$ which is  disjoint from $S_G$. Thus, let us concentrate on proving that $S_G$ is stationary.
        
        Let $\dot{C}$ be an $\dot{\mathbb{NR}}(\alpha)$-name, forced by some condition $p\in \mathbb{NR}(\alpha)$ to be a club in $\alpha$. Fix a winning strategy $\tau$ for Player II in the game $G_\alpha(\mathbb{NR}(\alpha))$ (such a strategy exists by Lemma \ref{Claim: NR is uniformly strategically closed}). Let $N$ be an elementary substructure of $H_{\chi}$ for some large enough $\chi$, such that:
        \begin{enumerate}
            \item $|N|< \alpha$.
            \item $\beta^* = \sup(N\cap \alpha)$ is an inaccessible cardinal below $\alpha$.
            \item $N$ is closed under $<\beta^*$-sequences of its elements.
            \item $\alpha, \mathbb{NR}(\alpha), p, \tau , \dot{C}\in N$.
        \end{enumerate}
        Such $N$ and $\beta^*$ can be constructed by standard arguments, using the fact that $\alpha$ is Mahlo. Next, let us construct sequences:
        \begin{itemize}
            \item $\la p_i \colon i<\beta^* \ra\subseteq \mathbb{NR}(\alpha)\cap N$, a decreasing sequence of conditions. 
            \item $\la \beta_i \colon i<\beta^* \ra\subseteq N\cap \alpha$ an increasing sequence of ordinals, which is unbounded in $\beta^*$, such that for every $i<\beta^*$,
            $$p_{i+1}\Vdash \exists \delta\in (\beta_{i}, \beta_{i+1})\cap \dot{C}.$$
        \end{itemize}
        
        We carry out the construction so that every strict initial segment of the above sequences belongs to $N$ (although the sequences themselves lie in $V$). We also maintain the inductive assumption that for every $i<\beta^*$, the sequence $\langle p_j \colon j<i\rangle$ is the sequence of moves played by Player II in a partial run of the game $G_\alpha(\mathbb{NR}(\alpha))$, where Player II follows their winning strategy $\tau$, and all conditions chosen so far by the players belong to $N$.

        Let $p_0 = p$ and $\beta_0 = 0$. Suppose that $i<\beta^*$ and that $\langle p_j, \beta_j \colon j<i\rangle$ have been constructed. Since $N$ is closed under $<\beta^*$-sequences of its elements, $\langle p_j \colon j<i\rangle \in N$, and it is the sequence of moves played by Player II in $G_\alpha(\mathbb{NR}(\alpha))$ according to $\tau$.

        If $i$ is limit, Player II may follow $\tau$ to choose a lower bound $p_i \in N$ of $\langle p_j \colon j<i\rangle$. Let $\beta_i = \sup_{j<i} \beta_j$.

        Assume $i = i'+1$ is successor. Let $q \leq p_{i'}$ in $N$ be such that $q$ decides $\min\big(\dot{C} \setminus (\beta_{i'}+1)\big)$. By elementarity, the decided value belongs to $N$. Let $\beta_i \in N \cap \alpha$ be a regular cardinal above both the decided value and $\max(\dom(q), \beta_{i'})$.\footnote{By elementarity, $N$ satisfies that every ordinal below $\alpha$ has a regular cardinal between it and $\alpha$, so such a $\beta_i$ exists.} Extend $q$ further to a condition $q^*$ with domain $\beta_i+1$ such that:
\begin{itemize}
            \item $q^*\restriction \dom(q) = q$.
            \item $q^*\restriction [\dom(q), \beta_{i}) = 0$.
            \item $q^*(\beta_{i}) = 1$.
        \end{itemize}
        Note that $q^*\in N$ is a legitimate condition since $\beta_{i}$ is a regular cardinal, so for every inaccessible cardinal $\lambda<\beta_{i}$, $q^*$ assigns the value $``1"$ only to ordinals of cofinality $\geq \lambda$. 
        Finally, let Player I pick $q^*$ as their next move in the game $G_\alpha(\mathbb{NR}(\alpha))$, and let $p_{i}\in N$ be the response of Player II when playing according to $\tau$. 

        This concludes the inductive construction. Since $\la p_i \colon i<\beta^* \ra\in V$ consists of the $\beta^*$ first rounds in a run of the game $G_\alpha(\mathbb{NR}(\alpha))$ according to $\tau$, there exists a condition which extends each of the conditions $\la p_i \colon i<\beta^*\ra$. In particular, $p^* = \bigcup_{i<\beta^*} p_i$ is a condition in $\mathbb{NR}(\alpha)$. By our construction, $p^*$ forces that $\beta^*$ is a limit point of $\dot{C}$. Let $p^{**} = p^*\cup \{ (\beta^*, 1) \}$. Then $p^{**}\in \mathbb{NR}(\alpha)$ since $\beta^*$ is regular, $p^{**}$ extends $p$, and $p^{**}\Vdash \beta^* \in \dot{C}\cap S_{\dot{G}}$.
    \end{proof}

    It will be useful that $\mathbb{NR}(\alpha)$ has arbitrarily closed dense subsets, a fact that it was also observed in \cite{ApterStrong}.

    \begin{lemma}\label{Lemma: NR_gamma is gamma directed}
      Let $\gamma<\alpha$ be a regular cardinal. Then $\mathbb{NR}_{\gamma}(\alpha)$ is a  $\gamma$-directed closed dense-open subset of $\mathbb{NR}(\alpha)$.
    \end{lemma}

    \begin{proof}
        The fact that $\mathbb{NR}_\gamma(\alpha)$ is a dense open subset of $\mathbb{NR}(\alpha)$ is clear. Let us argue that $\mathbb{NR}_\gamma(\alpha)$ is $\gamma$-directed closed. Assume $\delta<\gamma$ and $\la q_i \colon i<\delta \ra$ is a directed set of conditions in $\mathbb{NR}_\gamma(\alpha)$. Let $q^* = \bigcup_{i<\delta} q_i$. We argue that $q^*\in \mathbb{NR}_{\gamma}(\alpha)$ and from this it follows that $q^*$ extends each of the conditions $q_i$ for $i<\delta$. Let $\beta= \sup_{i<\delta} \dom(q_i)$. Clearly, $q^*\colon \beta\to 2$ is a well-defined function whose restriction to $S^\beta_{< \gamma}\setminus \gamma$ is identically $0$. It remains to see that for every cardinal $\lambda\leq\beta$ of uncountable cofinality, there exists a club $C\subseteq \lambda$ such that $q^*\restriction C$ is identically $0$. This is clear if $\lambda \leq \dom(q_i)$ for some $i<\delta$. Thus, it remains to check the case where $\lambda = \beta$ and $\dom(q_i)<\lambda$ for all $i<\delta$. In particular, $\lambda>\gamma$ and $\cf(\lambda)\leq \delta$. Since $\delta<\gamma$, we can find a club  $C\subseteq \lambda$ consisting of ordinals above $\gamma$, of cofinality strictly below $\gamma$. By the definition of the forcing, $q^*\restriction C$ is identically $0$. Overall, this proves that $q^*\in \mathbb{NR}_\gamma(\alpha)$, as desired.
    \end{proof}

    \begin{notation}
        Assume that $\kappa_0 < \kappa$ and  $\kappa$ is Mahlo. Let  $\mathbb{NR}( \kappa_0, \kappa )$ be the nonstationary support iterated forcing $\la \po_\alpha, \dot{\qo}_\alpha \colon \alpha<\kappa \ra$, where for  $\alpha<\kappa$,
        \begin{enumerate}
            \item If $\alpha \leq \kappa_0$ or $\alpha$ is not measurable in $V$, $\dot{\qo}_\alpha$ is forced to be trivial.
            \item Else, $\dot{\qo}_\alpha$ is forced to be $\left(\mathbb{NR}(\alpha)\right)^{V^{\po_\alpha}}$.
        \end{enumerate}
    \end{notation}

    We summarize below the basic properties of the forcing $\mathbb{NR}(\kappa_0, \kappa)$:

    \begin{lemma}\label{lemma: properties of NR(kappa0,kappa)}
        Denote $\po = \mathbb{NR}(\kappa_0, \kappa)$. Then:
        \begin{enumerate}
            \item \label{clause: basic properties of NR 1} $\po$ is $\mu$-strategically closed, being $\mu$  the first measurable  above $\kappa_0$. 
            \item \label{clause: basic properties of NR 2} $\po$ has a $\kappa_0$-directed closed dense subset.
            \item \label{clause: basic properties of NR 3} Assuming {\rm{GCH}}, $\po$ preserves cardinals and preserves {\rm{GCH}}.
        \end{enumerate}
    \end{lemma}

    \begin{proof}
        Clause \eqref{clause: basic properties of NR 1} follows, by an easy argument,  from Lemma \ref{Claim: NR is uniformly strategically closed}. 
        Clause \eqref{clause: basic properties of NR 2} follows from the fact that 
         $$D = \{ p\in \mathbb{NR}(\kappa_0, \kappa) \colon \forall \alpha\in \supp(p) \left( p\restriction \alpha \Vdash p(\alpha)\in \dot{\mathbb{NR}}_{\kappa_0}(\alpha) \right) \}$$
         is $\gamma$-directed closed by Lemma \ref{Lemma: NR_gamma is gamma directed}.\footnote{Note that $D$ is not open since coordinates can be added to the support when extending a condition.} Clause \eqref{clause: basic properties of NR 3} follows from Lemma \ref{Lemma: nonstatsupport iterations preserving GCH}.
    \end{proof}

    The following technical lemma will be useful in the proof of Theorem \ref{thm: Strong compactness after iterating nonreflecting stat sets}. It provides an example of a fusion argument, similar to that of Lemma \ref{Lemma: fusion lemma}, in which the dense sets $\langle d(\alpha) \colon \alpha<\kappa\rangle$ are not necessarily open.

    \begin{lemma}\label{lem: technical lemma for the proof of strong compactness}
        Denote $\po = \mathbb{NR}(\kappa_0, \kappa)$. Let $D\subseteq \po$ be the set of conditions $p\in \po$ for which there exists a club $C\subseteq \kappa$ such that for every $\alpha\in C$ and for every $\beta\in \supp(p)\setminus (\alpha+1)$,
        $$ p\restriction \beta \Vdash p(\beta) \in \dot{\mathbb{NR}}_{\alpha^{++}}(\beta). $$
        Then $D$ is a dense subset of $\po$.
    \end{lemma}

    \begin{proof}
        The proof is a modified and simplified version of Lemma \ref{Lemma: fusion lemma}.\footnote{A naive attempt to deduce it directly from Lemma \ref{Lemma: fusion lemma} would be to apply the lemma to the dense sets $d(\alpha)\subseteq \po$ consisting of conditions $p$ such that for every $\beta\in \supp(p)\setminus (\alpha+1)$, $p\restriction \beta \Vdash p(\beta)\in \dot{\mathbb{NR}}_{\alpha^{++}}(\beta)$. However, these sets $d(\alpha)$ are not open.}
        Fix a condition $p\in \po$. Let $C\subseteq \kappa$ be a club in $\kappa$ disjoint from $\supp(p)$. Let us construct sequences:
        \begin{itemize}
            \item $\la p_i \colon i<\kappa \ra$ a decreasing sequence of conditions in $\po$.
            \item $\la \alpha_i \colon i<\kappa \ra$ a continuous, increasing cofinal sequence in $\kappa$, such that each $\alpha_i$ belongs to $C$.
        \end{itemize}
The construction is done in such a way that the following properties hold:
    \begin{itemize}
        \item For every $i<\kappa$, $\supp(p_i) = \supp(p)$.
        \item For every $i<j<\kappa$, $p_{i}\restriction \alpha_i+1 = p_j \restriction \alpha_i+1$.
        \item For every $i<\kappa$ and $\beta\in \supp(p)\setminus (\alpha_{i}+1)$, $p_i\restriction \beta$ forces that:
        \begin{itemize}
            \item $p_i(\beta)\in \dot{\mathbb{NR}}_{(\alpha_{i})^{++}}(\beta)$.
            \item there exists $\gamma_i(\beta) \in (\alpha_i, \beta)$ such that $\dom(p_{i}(\beta) )= \gamma_i(\beta)+1$ and $p_{i}(\beta)(\gamma_i(\beta)) = 0$.
        \end{itemize}
        
    \end{itemize}
    
    Start the construction by letting $p_0 = p$, $\alpha_0 = \min(C)$. Assume that $p_i, \alpha_i$ have been constructed for some $i<\kappa$. Let $\alpha_{i+1} = \min(C \setminus \alpha_i+1)$. It's not hard to construct $p_{i+1}\leq p_i$ such that:
    \begin{itemize}
        \item $\supp(p_{i+1})= \supp(p)$.
        \item $p_{i+1}\restriction (\alpha_{i+1}+1)= p_i \restriction (\alpha_{i+1}+1)$.
        \item For every $\beta\in \supp(p)\setminus (\alpha_{i+1}+1)$, $p_{i+1}\restriction \beta$ forces that:
        \begin{itemize}
            \item $p_{i+1}(\beta)\in \dot{\mathbb{NR}}_{(\alpha_{i+1})^{++}}(\beta).$
            \item There exists $\gamma_{i+1}(\beta)\in (\alpha_{i+1}, \beta)$ such that $\dom(p_{i+1}(\beta)) = \gamma_{i+1}(\beta)+1$, and $p_{i+1}(\beta)(\gamma_{i+1}(\beta)) =0$. 
        \end{itemize}
    \end{itemize}
    Next, consider the limit case. Suppose that $i<\kappa$ and $p_j, \alpha_j$ were constructed for every $j<i$. Let $\alpha_i = \sup_{j<i}\alpha_j$. Define $p_i$ such that:
    \begin{itemize}
        \item $\supp(p_i)= \supp(p)$.
        \item $p_{i}\restriction \alpha_i = \bigcup_{j<i} p_j\restriction \alpha_j$.
        \item $p_{i}\restriction (\alpha_{i}+1)$ forces that $p_{i}\setminus (\alpha_i+1)$ extends all the conditions $\la p_j\setminus (\alpha_i+1) \colon j<i \ra$. 
        \item For every $\beta\in \supp(p)\setminus (\alpha_i+1)$, $p_{i}\restriction \beta$ forces that:
        \begin{itemize}
            \item $p_i(\beta)\in \dot{\mathbb{NR}}_{(\alpha_i)^{++}}(\beta).$
            \item There exists an ordinal $\gamma_i(\beta)\in (\alpha_i, \beta)$ such that $\dom(p_i(\beta)) = \gamma_i(\beta)+1$ and $p_i(\beta)(\gamma_i(\beta)) = 0$.
        \end{itemize}
    \end{itemize}
    Constructing $p_i$ for $i$ limit is not trivial as in the successor step, so let us justify why such a condition $p_i\in \po$ can be found. Note that $p_{i}\restriction \alpha_i$ is already fully determined, and $\alpha_i$ itself is outside the support of $p_i$. Thus, we concentrate on constructing $p_i\setminus (\alpha_i+1)$. Assume that $\beta\in \supp(p)\setminus (\alpha_i+1)$ and $p_i\restriction \beta$ has already been constructed. Work in a generic extension $V^{\po_\beta}$ that includes $p_i\restriction \beta$ in the generic, and let us construct $p_i(\beta)$ there. Note first that the sequence of conditions $\la p_j(\beta) \colon j<i \ra$ has a natural lower bound in $\mathbb{NR}(\beta)$, $$x = \bigcup_{j<i}p_j(\beta).$$
    We argue that $x$ is a legitimate condition in $\mathbb{NR}(\beta)$. Denote $\gamma^*:=\dom(x) = \sup_{j<i} \gamma_i(\beta)$. It suffices to prove that, for every regular cardinal $\lambda \leq \gamma^*$ of uncountable cofinality, there exists a club $D\subseteq \lambda$ such that $x\restriction D$ is identically $0$. Indeed, if $\lambda< \gamma^*$ this is clear, and if $\gamma^*$ itself is regular and $\lambda = \gamma^*$, then $\la \gamma_j(\beta) \colon j<i \ra$ is a club in $\gamma^*$ on which $x$ vanishes. 
    
    Overall, $x\in \mathbb{NR}(\beta)$, and by further extending $x$ inside $\mathbb{NR}(\beta)$, we can ensure that $x\in \mathbb{NR}_{(\alpha_i)^{++}}(\beta)$ and $\dom(x) = \gamma_i(\beta)+1$ for some $\gamma_i(\beta)\in (\alpha_i, \beta)$ with $p_i(\beta)(\gamma_i(\beta)) = 0$. Let $p(\beta)$ be forced by $p\restriction \beta$ to be the condition $x$ constructed above.
    
    This concludes the inductive construction. Finally, let 
    $$q = \bigcup_{i<\kappa} p_i\restriction \alpha_i \in \po.$$
    Note that $\supp(q) = \supp(p)$ and $q$ extends each of the conditions $\la p_i \colon i<\kappa \ra$. Furthermore, for every $i<\kappa$ and $\beta\in \supp(q)\setminus (\alpha_i+1)$, 
    $$ q\restriction \beta \Vdash q(\beta) \leq p_i(\beta) \in \mathbb{NR}_{(\alpha_{i})^{++}}(\beta).$$
    Thus, $q$ is as desired, as witnessed by the club $\{ \alpha_i \colon i<\kappa \}$.
    \end{proof}

    \begin{theorem}\label{thm: Strong compactness after iterating nonreflecting stat sets} 
        Assume {\rm{GCH}}. Assume that $\kappa$ is a supercompact cardinal, $\kappa_0<\kappa$, and let $\po = \mathbb{NR}(\kappa_0, \kappa)$. Suppose that $\lambda > \kappa$ is a measurable cardinal, and there are no measurable cardinals in the interval $(\kappa, \lambda)$. Then:
        \begin{enumerate}
            \item \label{clause: NR sometimes make supercompact into strongly compact 1} $\kappa$ is $\lambda$-strongly compact after forcing with $\mathbb{NR}(\kappa_0, \kappa)$.
            \item  \label{clause: NR sometimes make supercompact into strongly compact 1.5} There are no measurable cardianls in  $(\kappa_0, \kappa)$ after forcing with $\mathbb{NR}(\kappa_0, \kappa)$
            \item \label{clause: NR sometimes make supercompact into strongly compact 2} If $\lambda$ is strongly compact, then $\kappa$ remains so after forcing with $\mathbb{NR}(\kappa_0, \kappa)$.
        \end{enumerate}
    \end{theorem}

    For the proof of clause \eqref{clause: NR sometimes make supercompact into strongly compact 2} in Theorem \ref{thm: Strong compactness after iterating nonreflecting stat sets}, we will use Ketonen's characterization of strongly compact cardinals, which reads as follows:

    \begin{theorem}[{Ketonen \cite{Ketonen}; Usuba \cite[Lemmas 2.2 and 2.12]{usuba2021note}}]\label{thm: Ketonen's charachterization for strong compactness}
     Let $\kappa$ be a regular uncountable cardinal, and let $\lambda> \kappa$ be an inaccessible cardinal. Then the following are equivalent:
        \begin{enumerate}
            \item $\kappa$ is $\lambda$-strongly compact.
            \item For every regular cardinal $\mu \in [\kappa, \lambda]$, there exists a $\kappa$-complete uniform ultrafilter\footnote{An ultrafilter $U$ on a regular cardinal $\mu$ is uniform if every set in $U$ has size $\mu$.} on $\mu$.
        \end{enumerate}
    \end{theorem}

\begin{proof}[Proof of Theorem \ref{thm: Strong compactness after iterating nonreflecting stat sets}]

We begin by proving clause \eqref{clause: NR sometimes make supercompact into strongly compact 1}. 
Let $W$ be a fine, normal measure on $\mathcal{P}_\kappa(\lambda)$. 
Fix normal measures,  $U_0$ on $\kappa$ and $U_1$ on $\lambda$, both of Mitchell order $0$.

\smallskip

Perform the iterated ultrapower associated with the measure $U_0$, followed by the image of $U_1$, followed by the image of $W$. More formally, let:
\begin{itemize}
    \item $j_{U_0}\colon V\rightarrow M_0\simeq \Ult(V,U_0)$ be the ultrapower embedding associated with $U_0$. Denote $\kappa^* = j_{U_0}(\kappa)$. 
    \item $j^{M_0}_{j_{U_0}(U_1)}\colon M_U \rightarrow M_1 \simeq \Ult( M_0, j_{U_0}(U_1) )$ be the ultrapower embedding given by $j_{U_0}(U_1)$ over $M_0$.\footnote{
     $\lambda = j_{U_0}(\lambda)$ as $\lambda> \kappa$ is inaccessible. Thus $j_{U_0}(U_1)$ is a measure on $\lambda$ in $M_{0}$.} Denote $\lambda^* = j^{M_0}_{j_{U_0}(U_1)}(\lambda)$. 
    \item $j_{1} = j^{M_0}_{j_{U_0}(U_1)} \circ j_{U_0}\colon V\to M_{1}$.
    \item $i = j^{M_1}_{j_1(W)}\colon M_1\rightarrow M\simeq \Ult(M_1, j_1(W))$ the ultrapower embedding associated with $j_1(W)$ over $M_1$. 
    \item $j = i\circ j_1 \colon V\rightarrow M$.
\end{itemize}
Our goal is to show that $j$ lifts, after forcing with $\mathbb{P}:=\mathbb{NR}(\kappa_0, \kappa)$, to an elementary embedding that witnesses $\lambda$-strong compactness of $\kappa$ in $V[G]$, where $G\s \mathbb{P}$ is $V$-generic. Since there are no measurable cardinals in the interval $(\kappa, \lambda)$ in $V$, the forcing $j(\po)$ factors as follows in $M$:
$$ j(\po) = {j(\po)_{\kappa^*}} \ast {\mathbb{NR}(\kappa^*)} \ast {\mathbb{NR}(\lambda^*)} \ast {\mathbb{NR}(\lambda^*, j(\kappa))}.$$
We will construct generic objects for each of the factors.

Start with $j(\po)_{\kappa^*}$. By Lemma~\ref{Lemma: lifting elementary embeddings with fusion}, $j_{U_0}\colon V\rightarrow M_0$ lifts (inside $V[G]$) to $$j_{U_0}\colon V[G]\rightarrow M_0[G\ast (j_{U_0}[G]\setminus \kappa+1)].$$

(Note that the forcing used at stage $\kappa$ of $j_{U_0}(\po)$ is trivial, as $U$ is assumed to have  trivial Mitchell rank. This is why we obtain the above lifting).

Denote $H:=G\ast (j_{U_0}[G]\setminus \kappa+1)$. Clearly, $H$ is $j(\mathbb{P})_{\kappa^*}$-generic over $M_1$ and $M$ as well. Both $M_1[H]$ and $M[H]$ are closed under $\lambda$-sequence of their elements inside $M_0[H]$, and $M[H]$ is closed under $\lambda^*$-sequences in $M_1[H]$.

\smallskip

Proceed now with $\mathbb{NR}(\kappa^*)$. Let us construct $h_0\in V[G]$ which is $\mathbb{NR}(\kappa^*)$-generic over the models $M_0[H], M_1[H]$ and $M[H]$. {Note that $$(\mathbb{NR}(\kappa^*))^{M_0[H]} = (\mathbb{NR}(\kappa^*))^{M_1[H]} = (\mathbb{NR}(\kappa^*))^{M[H]}$$ since $M_1[H]$ and $M[H]$ are closed under $\lambda$-sequences in $M_0[H]$.}

For each $p\in G$,  elementarity implies that  $j(\supp(p))$ is  nowhere stationary below $j(\kappa)$. In particular, $j(\supp(p))\cap \kappa^*$ is nonstationary in $\kappa^*$, which entails $\kappa^*\notin j(\supp(p))$.\footnote{Indeed, let $C\s \kappa^*$ be a club disjoint from $j(\supp(p))\cap \kappa^*=\supp(p)$. Since $C\s j(C)$ and $j(C)$ is a club on $j(\kappa^*)$, it follows that $\kappa^*\in j(C)$. Since $j(C)$ is disjoint from $j(\supp(p))$ we obtain the desired conclusion.} 
This observation  indicates that there is no need to construct $h_0$ below a master condition; putted in different  words, any $h_0\in V[G]$ which is generic over $M[H]$ would work for lifting purposes. Additionally, let us observe  that since both $M_1[H]$ and $M[H]$ are closed under $\lambda$-sequences in $M_0[H]$, it suffices to find such $h_0\in V[G]$ which is generic over $M_0[H]$. We construct $h_0$ as follows: Recall that $M_0[H]$ is closed under $\kappa$-sequences inside $V[G]$, and thus, by Lemma \ref{Claim: NR is uniformly strategically closed}, $(\mathbb{NR}(\kappa^*))^{M_0[H]}$ is $\kappa^+$-strategically closed in $V[G]$. Also, $V[G]$ lists in order-type $\kappa^+$ all dense subsets of $(\mathbb{NR}(\kappa^*))^{M_0[H]}$ inside the model $M_0[H]$. As a result, $h_0\in V[G]$ can be constructed by meeting all those dense sets, one-by-one, in a run of the game $G_{\kappa^+}(\mathbb{NR}(\kappa^*))$ in which Player II plays according to their winning strategy. Note that both $M_{1}[H*h_0]$ and $M[H*h_0]$ remain closed under $\lambda$-sequences inside $M_0[H\ast h_0]$, and that  $M[H*h_0]$ stays closed under $\lambda^*$-sequences in  $M_1[H*h_0]$.

\smallskip

Next, proceed to $\mathbb{NR}(\lambda^*)$. Let us construct $h_1\in M_0[H\ast h_0]$ which is $\mathbb{NR}(\lambda^*)$-generic over both $M_{1}[H*h_0]$ and $M[H*h_0]$. Once again, note that 
$$(\mathbb{NR}(\lambda^*))^{M_1[H*h_0]}= (\mathbb{NR}(\lambda^*))^{M[H*h_0]}$$
as $M[H*h_0]$ is closed under $\lambda^*$-sequences in $M_1[H*h_0]$.

This part of the construction posses an addition caveat: It is not true anymore that $\lambda^*\notin \supp(j(p))$ for conditions $p\in G$. Therefore, in order to apply Silver's lifting criterion at the end of the construction, we will need to ensure that the generic $h_1$ is chosen so that  $j[G]\restriction (\lambda^*+1)\s H*h_0*h_1$. We do this by constructing a suitable master condition. The main obstacle is that the forcing $\mathbb{NR}(\lambda^*)$ is not sufficiently directed closed; by Lemma \ref{Lemma: NR_gamma is gamma directed} it has a dense, $\zeta$-directed closed subset for arbitrarily high $\zeta<\lambda^*$, but it's unclear that conditions of the form $i(q)(\lambda^*)$ for $q\in H$ make it into the relevant dense directed closed subset. We overcome this issue by isolating a dense subset $D\subseteq j_1(\po)$, whose conditions are transferred through $i$ to a sufficiently directed closed dense subset of $j(\po)$.   

By Lemma \ref{lem: technical lemma for the proof of strong compactness}, there exists a dense subset $D\subseteq j_1(\po)$ consisting of all conditions $q\in j_1(\po)$ for which there exists a club $C\subseteq \kappa^*$, such that, for every $\alpha\in C$ and for every $\beta\in \supp(q)\setminus ( \alpha+1 )$,
$$ q\restriction \beta \Vdash q(\beta)\in \dot{\mathbb{NR}}_{\alpha^{++, M_1}}(\beta). $$
Given $q\in H\cap D$, let $C\subseteq \kappa^*$ be the club witnessing the fact that $q\in D$. Since $\crit(i) = \kappa^*$, $\kappa^*\in i(C)$. Thus, if $\lambda^*\in \supp(i(q))$, then 
$$q\restriction \lambda^* \Vdash q(\lambda^*)\in \dot{\mathbb{NR}}_{(\kappa^*)^{++, M}}(\lambda^*).$$
Working in $M[H*h_0]$, consider the set $$A = \{ i(q)(\lambda^*) \colon q\in (H\cap D) \land \lambda^*\in \supp(i(q)) \}.$$ Since $i\restriction j_1(\po)\in M$, $A\in M[H*h_0]$. Also $A\subseteq \mathbb{NR}_{(\kappa^*)^{++}}(\lambda^*)$ and $|A| = (\kappa^*)^+$. Since $A$ is a directed set of conditions and $\mathbb{NR}_{(\kappa^{*})^{++}}(\lambda^*)$ is a $(\kappa^{*})^{++}$-directed closed poset, there exists in $M[H*h_0]$ a lower bound $t\in \mathbb{NR}(\lambda^*)$ of all the conditions in $A$. In particular, by the density of $D$, $t$ is a lower bound of all the conditions in $\{ i(q)(\lambda^*) \colon q\in H \}$. Since $H$ is generated by  conditions of the form $j_1(p)$\footnote{$H$ was generated by conditions of the form $j_{U_0}(p)$. However, the critical point of $j_{U_0(U_1)}$ was way above the rank of $j_{U_0}(\mathbb{P})$ so $H=\{j_1(p)\mid p\in G\}$ is also generic.} for $p\in G$, we deduce that $t$ extends $j(p)(\lambda_*)$ for all $p\in G$. Finally, we construct $h_1\in M_0[G*h_0]$ which is $\mathbb{NR}(\lambda^*)$-generic over both $M_1[H*h_0]$ and $M[H*h_0]$ with $t\in h_1$. To do this, note that $\mathbb{NR}(\lambda^*)$ is $\lambda^{+,M_0}$-strategically closed in the sense of $M_0[H*h_0]$, since $M_1[H*h_0]$ and $M[H*h_0]$ are closed under $\lambda$-sequences in $M_0[H*h_0]$. Furthermore, $M_0[H*h_0]$ has a list of order-type $\lambda^{+,M_0}$ of all dense subsets of $\mathbb{NR}(\lambda^*)$. Thus, $h_1\in M_0[H*h_0]$ can be constructed by meeting all those dense subsets, one-by-one, in a run of the game $G_{\lambda^{+,M_0}}(\mathbb{NR}(\lambda^*))$ in which Player II plays according to their winning strategy, and Player I opens the game by picking the condition $t$. 

By standard arguments, $M[H*h_0*h_1]$ remains closed under $\lambda^*$-sequences in $M_1[H*h_0*h_1]$.

\smallskip

Finally, let us proceed with the last bit of the iteration, $\mathbb{NR}(\lambda^*, j(\kappa))$. Let us construct a suitable generic set $H^*\subseteq \mathbb{NR}(\lambda^*, j(\kappa))$ over $M[H*h_0*h_1]$ such that $i[H]\setminus (\lambda^*+1)\subseteq H^*$, where as usual
$$ i[H]\setminus (\lambda^*+1) := \{i(q)\setminus (\lambda^*+1)\colon q\in H\}.$$
In order to construct $H^*$, we first need to  construct a suitable master condition $s\in \mathbb{NR}(\lambda^*, j(\kappa))$, such that $s$ extends all the conditions in $i[H]\setminus (\lambda^*+1)$. We do this by the same technique that facilitated the construction of $t$ above. Let $D\subseteq j_1(\po)$ be the same dense set as above. Given $q\in H\cap D$, let $C\subseteq \kappa^*$ be the club witnessing the fact that $q\in D$. As above, $\kappa^*\in i(C)$, so for every $\beta\in \supp(i(q))\setminus ( \kappa^*+1 )$, 
$$i(q)\restriction\beta\Vdash i(q)(\beta)\in \dot{\mathbb{NR}}_{(\kappa^*)^{++}}(\beta).$$
In particular, $B = \{ i(q)\setminus (\lambda^*+1) \colon q\in H\cap D \}$ is contained in a $(\kappa^{*})^{++, M}$-directed closed subset of $\mathbb{NR}(\lambda^*, j(\kappa))$. Since $|H| = (\kappa^*)^{+, M}$, there exists a lower bound $s\in \mathbb{NR}(\lambda^*, j(\kappa))$ of $B$. The same condition $s$ is a lower bound of $i[H]\setminus (\lambda^*+1)$ because $D$ is dense in $j_1(\po)$. Thus, it remains to construct $H^*\in M_1[H*h_0*h_1]$ such that $H^*\subseteq \mathbb{NR}(\lambda^*, j(\kappa))$ is generic over $M[H*h_0*h_1]$ and $s\in H^*$. Since $M[H*h_0*h_1]$ is closed under $\lambda^*$-sequences inside $M_1[H*h_0*h_1]$, $\mathbb{NR}(\lambda^*, j(\kappa))$ is $(\lambda^*)^{+, M_1}$-strategically closed in $M_1[H*h_0*h_1]$. Also, $M_1[H*h_0*h_1]$ has an enumeration of order type $(\lambda^*)^{+, M_1}$ consisting of all dense subsets of $\mathbb{NR}(\lambda^*, j(\kappa))$. Thus, the generic $H^*\in M_1[H*h_0*h_1]$ can be constructed by meeting all those dense sets in a run of the game $G_{(\lambda^*)^{+, M_1}}( \mathbb{NR}(\lambda^*, j(\kappa)) )$ in which Player II plays according to their winning strategy and Player I opens the game by picking the condition $s$. Finally, note that $M[H*h_0*h_1*H^*]$ remains closed under $\lambda^*$-sequences in $M_1[H*h_0*h_1]$.

\smallskip

This concludes the construction of the generic extension $M[H*h_0*h_1*H^*]$ inside $V[G]$. Notice that we made sure through the construction that $$j[G]\subseteq H*h_0*h_1*H^*.$$
Therefore, by Silver's lifting criterion, $j$ lifts to an embedding $$j^*\colon V[G]\rightarrow M[H\ast h_0 \ast h_1 \ast H^*].$$ It remains to prove that $j^*$  witnesses $\lambda$-strong compactness of $\kappa$ in $V[G]$. Indeed, this follows by combining the following facts:
    \begin{itemize}
        \item $M[H\ast h_0\ast h_1\ast H^*]$ is closed under $\kappa$-sequences in $V[G]$.
        \item $j(\kappa)> \lambda$. 
        \item $j[\lambda]$ is covered by the set $S:=i[j_1(\lambda)] = i[\lambda^*]\in M[H\ast h_0 \ast h_1 \ast H^*]$, and $|S|^{M[H\ast h_0 \ast h_1 \ast H^*]}=\lambda^* <j(\kappa)$.
    \end{itemize}

    For clause \eqref{clause: NR sometimes make supercompact into strongly compact 1.5} of the theorem, observe that no cardinal is measurable if it carries a nonreflecting stationary subset. Hence every $V$-measurable cardinal in the interval $(\kappa_0, \kappa)$ ceases to be measurable after forcing with $\po$. Moreover, any measurable cardinal in $V[G]$ must already have been measurable in $V$ (Indeed, $\po$ has a gap at the least $V$-measurable cardinal and is trivial below it, so it cannot create new measurable cardinals). Thus, in $V[G]$, the interval $(\kappa_0,\kappa)$ does not contain measurable cardinals.
  
    Finally, let us prove clause \eqref{clause: NR sometimes make supercompact into strongly compact 2} of the theorem. By Theorem \ref{thm: Ketonen's charachterization for strong compactness}, it suffices to prove that, in $V[G]$, every regular cardinal $\mu\geq \kappa$ has a $\kappa$-complete uniform ultrafilter. By clause \eqref{clause: NR sometimes make supercompact into strongly compact 1}, $\kappa$ is $\lambda$-strongly compact in $V[G]$, so by Theorem \ref{thm: Ketonen's charachterization for strong compactness}, every regular cardinal $\mu\in [\kappa, \lambda]$ carries a $\kappa$-complete uniform ultrafilter in $V[G]$. Thus it remains to take care of regular cardinals above $\lambda$. Since $\lambda$ is strongly compact in $V$, every regular cardinal $\mu> \lambda$ carries a uniform, $\lambda$-complete ultrafilter $U_\mu\in V$. Since $|\po| = \kappa^+$, the same ultrafilter generates a uniform, $\lambda$-complete ultrafilter $U^*_\mu$ on $\mu$ in $V[G]$,\footnote{Indeed, for every $\eta<\lambda$ and a sequence $\la A_i \colon i<\eta \ra \subseteq U^*_\mu$ in $V[G]$, find a sequence $\la B_i \colon i<\eta \ra\subseteq U_\mu$ in $V[G]$ such that $B_i \subseteq A_i$ for all $i$. Let $S\in V$, $S\subseteq U_\mu$, $|S|<\lambda$ be a set that covers $\{ B_i\colon i<\eta \}$. Then $\bigcap S\subseteq \bigcap_{i<\eta} A_i$ and $\bigcap S\in U_\mu$, so $\bigcap_{i<\eta} A_i\in U^*_\mu$.} as desired.
\end{proof}

    The case where there are no measurable cardinals above $\kappa$ will also matter to us, and it can be proved using the same techniques as above. In fact, the proof is simpler since there is no need to take an additional ultrapower with a normal measure of order $0$ on $\lambda$.

\begin{theorem}\label{thm: Strong compactness after iterating nonreflecting stat sets when there are no measurable cardinals above} 
        Assume {\rm{GCH}}. Assume that $\kappa$ is a supercompact cardinal, $\kappa_0<\kappa$, and let $\po = \mathbb{NR}(\kappa_0, \kappa)$. Suppose that there are no measurable cardinals above $\kappa$. Then $\kappa$ is strongly compact after forcing with $\mathbb{NR}(\kappa_0, \kappa)$. Moreover, there are no measurable cardinals in the interval $(\kappa_0, \kappa)$ after forcing with $\mathbb{NR}(\kappa_0, \kappa)$.
    \end{theorem}

    \begin{proof}
        Fix a strong limit cardinal  $\lambda>\kappa$ with $\cf(\lambda)> \kappa$. 
        Fix a fine, normal measure $W$ on $\mathcal{P}_\kappa(\lambda)$. Let $U$ be a normal measure on $\kappa$ of Mitchell order $0$. Consider the iterated ultrapower, first with $U$ and then with the image of $W$, namely 
        $$j = j^{M_U}_{j_{U}(W)}\circ j_U\colon V\to M.$$
        Note that $\lambda = j_U(\lambda)$ is not measurable under the hypotheses of the theorem, so there is no need to take an ultrapower with a normal measure of order $0$ on $\lambda$. Denote $\kappa^* = j_U(\kappa)$, and factor 
        $$j(\po) = j(\po)_{\kappa^*}\ast \dot{\mathbb{NR}}(\kappa^*) \ast \dot{\mathbb{NR}}(\lambda, j(\kappa)).$$
        Now mimic the proof of Theorem \ref{thm: Strong compactness after iterating nonreflecting stat sets} to find suitable generics $H*h_0*H^*$ for $j(\po)$ over $M$, and lift $j$ to $j^*\colon V[G]\to M[H*h_0*H^*]$, witnessing $\lambda$-strong compactness of $\kappa$ in $V[G]$. We omit the details here since they are exactly the same as in the proof of Theorem \ref{thm: Strong compactness after iterating nonreflecting stat sets}.
        
        Since $\lambda$ was arbitrarily large, $\kappa$ is fully strongly compact in $V[G]$.
    \end{proof}

We conclude this subsection by characterizing normal measure on $\kappa$ after forcing with $\mathbb{NR}(\kappa_0, \kappa)$.

\begin{lemma} \label{lemma: normal measures after iterating NRs}
    Let $\kappa$ be a measurable cardinal and  $\kappa_0<\kappa$. Denote $\po = \mathbb{NR}(\kappa_0, \kappa)$. Let $G\subseteq \po$ be generic over $V$. Then $\kappa$ remains measurable in $V[G]$. Moreover:
    \begin{enumerate}
        \item For every normal measure $U\in V$ on $\kappa$ of Mitchell order $0$, $U$ generates a normal measure $U^*\in V[G]$ on $\kappa$.
        \item Every normal measure $W\in V[G]$ on $\kappa$ has the form $U^*$ for some normal measure $U\in V$ of Mitchell order $0$ on $\kappa$.
    \end{enumerate}
\end{lemma}

\begin{proof}
    Let $U\in V$ a normal measure on order $0$ on $\kappa$. Note that $\kappa$ is not measurable in $M_U\simeq \Ult(V, U)$. Thus, by Lemma \ref{Lemma: lifting elementary embeddings with fusion}, $j_{U_0}$ lifts to an embedding $j^*\colon V[G]\to M[G*(j_{U_0}[G]\setminus (\kappa+1))]$. Define $U^* = \{ X\subseteq \kappa \colon \kappa\in j^*(X)\}$. Then $U^*$ is a normal measure on $\kappa$ in $V[G]$. It's not hard to verify that $j^*$ is the ultrapower embedding associated with $U^*$. Let us argue that $U$ generates $U^*$ in $V[G]$. Assume that $X\in V[G]$ is a subset of $\kappa$ in $U^*$. Let $\dot{X}$ be a $\po$-name for $X$. Then, for some $p\in G$, $j_{U_0}(p)\Vdash \kappa\in j_{U_0}(\dot{X})$. It follows that $A = \{ \alpha<\kappa \colon p\Vdash \alpha\in X \}\in U$, and $A\subseteq X$.

    Conversely, assume that $W \in V[G]$ is a normal measure on $\kappa$. By Hamkins' Gap Forcing Theorem (Theorem \ref{thm: Gap forcing theorem}), we have $U := W \cap V \in V$. It is straightforward to verify that $U \in V$ is a normal measure on $\kappa$. We claim that $U$ must have Mitchell order $0$. Let $\Delta \subseteq \kappa$ denote the set of measurable cardinals in $V$, and suppose toward a contradiction that $\Delta \in U$. Then in particular $\Delta \in W$. Recall that no cardinal in $\Delta$ remains measurable in $V[G]$, since $\mathbb{P}$ adds a nonreflecting stationary subset to each such cardinal. Since $W$ is normal and $\Delta \in W$, it follows that $\kappa \in j_W(\Delta)$. Thus, in $\mathrm{Ult}(V[G], W)$, there exists a nonreflecting stationary subset $S \subseteq \kappa$.    However, $V[G]$ and $\mathrm{Ult}(V[G], W)$ agree on subsets of $\kappa$ and on their stationarity. Therefore, $S$ is a nonreflecting stationary subset of $\kappa$ in $V[G]$, contradicting the measurability of $\kappa$ in $V[G]$. 
    
    It follows that $U$ has Mitchell order $0$. Finally, since $U\subseteq W$ and $U$ generates an ultrafilter $U^*\in V[G]$, we deduce that $W = U^*$, as desired.
\end{proof}
    
\section{The main theorems}\label{sec:main section}
In this section we will apply the theory developed so far to prove the main theorems of the paper. The general theme (as indicated in the introduction) is to produce models where measurable cardinals beyond the current reach of the Inner Model Program (e.g., measurable limits of supercompacts, etc.) carry any prescribed amount of normal measures. All of our results are proved under the assumption that in the ground model, each relevant measurable cardinal carries exactly one normal measure with Mitchell order $0$.

\subsection{The first finitely many measurable cardinals}

We begin proving a variation of  the classical theorem of Kimchi and Magidor (unpublished) saying that, for each $n<\omega$, the first $n$ measurable cardinals can all be strongly compact and each of them carries any prescribed number of normal measures. 

\begin{theorem}\label{thm: IdentityCrises + Normal measures}
    Assume the $\mathrm{GCH}$ holds, there are $n<\omega$ supercompact cardinals $\langle \kappa_i\colon i<n\rangle$, there are no measurable cardinals above $\kappa_{n-1}$,  and each $\kappa_i$ has a unique normal measure of Mitchell order $0$. For every $i<n$, let $\tau_i\leq\kappa^{++}_i$ be a cardinal. Then there is a generic extension where {\rm{GCH}} holds, $\langle \kappa_i\colon i<n\rangle$ are the first $n$ strongly compact cardinals, the first $n$ measurable cardinals,  and each $\kappa_i$ carries exactly $\tau_i$ normal measures.
\end{theorem}
\begin{proof}
We first make the supercompact cardinals $\kappa_0,\ldots, \kappa_{n-1}$ indestructible under sufficiently directed closed forcings which preserve cardinals and GCH. We do this while simultaneously ensuring that for each $i<n$, $\kappa_i$ carries a unique normal measure of Mitchell order $0$. To that end, denote for each $i<n$, $\rho_i := (\kappa_{i-1})^+$ (if $i=0$, $\rho_0 := 0$) and define a finite iterated forcing 
$$\mathbb{L} = \langle \mathbb{L}_i, \dot{\mathbb{L}}(\kappa_i, \rho_i)\colon i<n\rangle$$
where $\dot{\mathbb{L}}(\kappa_i,\rho_i)$ is the canonical $\mathbb{L}_i$-name for the $\rho_i$-directed closed closed forcing which makes the supercompactness of $\kappa_i$ indestructible under $(\kappa_i)^+$-directed closed forcings, while maintaining the same number of normal measures of Mitchell order $0$ on $\kappa_i$ as in $V^{\mathbb{L}_i}$ (see Notation~\ref{notation: Laver}).

\begin{claim}
    $\mathbb{L}$ preserves cardinals, $\rm{GCH}$, and the supercompactness of each $\kappa_i$.  Moreover, each $\kappa_i$ remains indestructible under $(\kappa_i)^+$-directed-closed forcings and carries a unique normal measure of Mitchell order $0$ after forcing with $\mathbb{L}$.
\end{claim}
\begin{proof}[Proof of claim]
The fact that $\mathbb{L}$ preserves cardinals and the GCH follows from it being a finite iteration of forcings with these properties. Therefore, we focus on the supercompactness and indestructibility of each $\kappa_i$. 

Let $G\subseteq \mathbb{L}$ be generic over $V$. Fix $i<n$. We verify that $\kappa_i$ remains supercompact in $V[G]$. Factor the forcing as
$\mathbb{L} = \mathbb{L}_i \ast \dot{\mathbb{L}(\kappa_{i+1},\rho_{i+1})} \ast \mathbb{L}\setminus (i+1)$.
Since $\mathbb{L}_i \in H(\kappa_i)$, forcing with $\mathbb{L}_i$ does not affect the supercompactness of $\kappa_i$, and hence $\kappa_i$ remains supercompact in $V[G_i]$. In addition, because $\kappa_i$ has in $V$ a unique normal measure of Mitchell order $0$, the same holds in $V[G_i]$.  As GCH remains true in $V[G_i]$, Theorem \ref{theorem: Laver indestructibility+GCH} applies. It follows that in $V[G_{i+1}]$, the cardinal $\kappa_i$ is supercompact, carries a unique normal measure of Mitchell order $0$, and is indestructible under $(\kappa_i)^+$-directed closed forcings.

Now observe that the tail forcing $\mathbb{L}\setminus (i+1)$ is $(\kappa_i)^+$-directed closed. By indestructibility, $\kappa_i$ therefore remains supercompact in $V[G]$. Furthermore, no new normal measures on $\kappa_i$ are added, so in $V[G]$ it carries exactly the same normal measures as in $V[G_{i+1}]$; in particular, it still has a unique normal measure of Mitchell order $0$.

It remains to verify that $\kappa_i$ is indestructible in $V[G]$. Let $\qo \in V[G]$ be a $(\kappa_i)^+$-directed closed poset. Then the two-step forcing $\mathbb{L}\setminus (i+1) \ast \dot{\qo}$ is itself $(\kappa_i)^+$-directed closed. Since $\kappa_i$ was already indestructible under such forcings in $V[G_i]$, it follows that forcing with $\qo$ over $V[G]$ preserves the supercompactness of $\kappa_i$. This concludes the proof of the claim.
\end{proof}

To simplify the notation, let us assume that $V$ is already a generic extension by $\mathbb{L}$, namely, assume that GCH holds in $V$, $\kappa_0,\ldots, \kappa_{n-1}$ are supercompact cardinals, each $\kappa_i$ is indestructible under $(\kappa_i)^+$-directed closed forcings and each $\kappa_i$ carries a unique normal measure of Mitchell order $0$. Consider the product forcing
$$\mathbb{R} = \textstyle \prod_{i<n}(\dot{\mathbb{P}}^{\tau_i,I_i}\ast \dot{\mathbb{NR}}{(\kappa_{i-1},\kappa_i)})$$ where
\begin{itemize}
    \item $I_i$ is the set of inaccessible cardinals in $(\kappa_{i-1},\kappa_i)$ that are limits of strong cardinals (here and below, we let $\kappa_{i-1} = 0$ for $i=0$).
    \item If $\tau_i \leq (\kappa_i)^+$, $\mathbb{P}^{\tau_i,I_i}$ is the spaced splitting forcing from Lemma \ref{lemma: the spaced splitting forcing preserves supercompactness}.
    \item If $\tau_i = (\kappa_i)^{++}$, $\po^{\tau_i, I_i}$ is the forcing from Lemma \ref{lemma: kappa^++ many normal measures}.
    \item $\dot{\mathbb{NR}}{(\kappa_{i-1},\kappa_i)}$ is the nonstationary support iterated forcing which adds a  nonreflecting stationary subset to every measurable cardinal in the interval $(\kappa_{i-1}, \kappa_i)$, as defined in section \ref{Section: Forcing NR}.
\end{itemize}

We argue that in $V^{\mathbb{R}}$, the cardinals $\kappa_0, \ldots, \kappa_{n-1}$ are strongly compact, there are no other measurable cardinals, and each $\kappa_i$ carries exactly $\tau_i$ normal measures. This will be proved in the following claim, which  concludes the proof of the theorem:
\begin{claim}
Let $i<n-1$. After forcing over $V$ with 
    $$\mathbb{R}_{\geq i} = \textstyle \prod_{j\geq i}(\dot{\mathbb{P}}^{\tau_j,I_j}\ast \dot{\mathbb{NR}}{(\kappa_{j-1},\kappa_j)}),$$
    all the cardinals $\kappa_{i}, \ldots, \kappa_{n-1}$ are strongly compact, there are no other measurable cardinals above $\kappa_{i-1}$, and each cardinal $\kappa_{j}$ for $j\geq i$ carries exactly $\tau_i$ normal measures.

\end{claim}
\begin{proof}[Proof of claim]
Following the proof of Apter--Cummings \cite{ApterCummingsStrongCompactness}, we perform a  downwards induction. We first establish the claim for $i = n-1$, and the inductive step consists of proving the claim for $i$, assuming it holds for $i+1$. The theorem follows from the case $i=0$.

    Let us first take care of $\kappa_{n-1}$ ($\kappa_{n-1}$ is the top cardinal). Consider the forcing $$\mathbb{R}_{\geq n-1} = \dot{\po}^{\tau_{n-1}, I_{n-1}} \ast \dot{\mathbb{NR}}(\kappa_{n-1}, \kappa_n)$$
    over $V$. 
    By Lemmas \ref{lemma: the spaced splitting forcing preserves supercompactness} and \ref{lemma: the splitting forcing preserves Mitchell order}, in $V^{\dot{\po}^{\tau_{n-1}, I_{n-1}}}$, GCH holds, $\kappa_{n-1}$ remains supercompact and carries exactly $\tau$ normal measures of Mitchell order $0$. Since there are no measurable cardinals above $\kappa_{n-1}$ in $V^{\dot{\po}^{\tau_{n-1}, I_{n-1}}}$, we can apply Theorem \ref{thm: Strong compactness after iterating nonreflecting stat sets when there are no measurable cardinals above} to deduce that in $V^{\mathbb{R}_{\geq n-1}}$, $\kappa_{n-1}$ is a strongly compact cardinal. By Lemma \ref{lemma: normal measures after iterating NRs}, $\kappa_{n-1}$ carries exactly $\tau_{n-1}$ normal measures in $V^{\mathbb{R}}$.

    Next, we take care of $\kappa_i$ for $i<n-1$. Consider the forcing
    $$\mathbb{R}_{\geq i} = \textstyle \prod_{j>i}(\dot{\mathbb{P}}^{\tau_j,I_j}\ast \dot{\mathbb{NR}}{(\kappa_{j-1},\kappa_j)}),$$
    and factor it to the form 
    $$ \mathbb{R}_{\geq i} = \mathbb{R}_{>i} \times \left( \dot{\po}^{\tau_{i}, I_{i}} \ast \dot{\mathbb{NR}}(\kappa_{i-1}, \kappa_{i}) \right). $$
    The follwing properties hold after forcing over $V$ with $\mathbb{R}_{>i}$:
    \begin{itemize}
        \item $\kappa_i$ is supercompact, since $\kappa_i$ is indestructible in  in $V$ under $(\kappa_i)^+$-directed closed forcings, and $\mathbb{R}_{>i}$ is such a forcing.
        \item $\kappa_{i+1}, \ldots, \kappa_{n-1}$ are strongly compact, there are no other measurable cardinals above $\kappa_i$, and each $\kappa_j$ (for $j>i$) carries exactly $\tau_j$ normal measures. This follows from the induction hypothesis.
        \item GCH holds, since $\mathbb{R}_{>i}$ preserves GCH as a finite product of forcings which preserve GCH.
        \item $\kappa_i$ has a unique normal measure of Mitchell order $0$, since this is the case in $V$, and  $\mathbb{R}_{>i}$ is a $(\kappa_{i})^{++}$-closed poset.
    \end{itemize}
    By Theorem \ref{lemma: the spaced splitting forcing preserves supercompactness}, all the above bulleted points, except for the last one, are carried on to $V^{\mathbb{R}_{>i}\times \dot{\po}^{\tau_{i}, I_i}}$. The last bulleted point is now replaced with $\kappa_i$ carrying exactly $\tau_i$ normal measures of Mitchell order $0$. Finally, by Theorem \ref{thm: Strong compactness after iterating nonreflecting stat sets}, in $V^{\mathbb{R}_{\geq i}}=V^{ \mathbb{R}_{>i}\times \left( \dot{\po}^{\tau_{i}, I_i} \ast \dot{\mathbb{NR}}(\kappa_{i-1}, \kappa_i) \right) }$, $\kappa_i$ becomes strongly compact and the least measurable cardinal above $\kappa_{i-1}$. Moreover, By Lemma \ref{lemma: normal measures after iterating NRs}, $\kappa_i$ has exactly $\tau_i$ normal measures. The forcing $\mathbb{P}^{\tau_i, I_i} * \dot{\mathbb{NR}}(\kappa_{i-1}, \kappa_i)$ is small relative to the cardinals $\kappa_{i+1}, \ldots, \kappa_{n-1}$. Hence their strong compactness, the fact that they are the only measurable cardinals above $\kappa_i$, and the number of normal measures they carry all remain unchanged from $V^{\mathbb{R}_{>i}}$. Thus, overall, we obtain that the following properties holds after forcing with $\mathbb{R}_{\geq i}$: $\kappa_i, \kappa_{i+1}, \ldots ,\kappa_{n-1}$ are strongly compact cardinals, there are no measurable cardinals above $\kappa_{i-1}$, and each $\kappa_j$ (for $j\geq i$) has exactly $\tau_j$ normal measures. This concludes the inductive argument.
\end{proof}
This concludes the proof. 
\end{proof}
hree-step iteration consisting of: (1) the poset $\mathbb{L}_i$ that makes the supercompactness of $\kappa_i$ indestructible under $(2^{\kappa_i})^+$-directed-closed forcing; (2) the splitting poset $\mathbb{P}^{\tau_i}$ from page~\pageref{page: splitting forcing} starting the iteration past $(2^{\kappa_i})^+$; (3) the nonstationary support iterations which adds a non-reflecting stationary set without fixed cofinality to every measurable in $(\kappa_{i-1},\kappa_i)$.

\subsection{The first measurable above a supercompact cardinal}

We now turn to the least measurable cardinal above a supercompact cardinal. The Friedman–Magidor technique \cite{FriedmanMagidor} does not allow one to control the number of normal measures on such a cardinal, since it is carried out over $L[U]$.\footnote{More precisely, the Friedman–Magidor construction is performed over an arbitrary model of GCH; however, the first step is a class-forcing iteration coding the universe by a real, thereby transferring it to $L[U,r]$ for some real $r$. This procedure destroys supercompact cardinals.} However, the splitting forcing can be used to control the number of normal measures on the least measurable cardinal above an indestructible supercompact cardinal.
\begin{theorem}\label{thm: above supercompact}
    Assume the $\mathrm{GCH}$ holds,  $\kappa$ is a supercompact cardinal, and $\lambda$ is the first measurable cardinal above $\kappa$. In addition, suppose that $\lambda$ has a unique normal measure. Then, for each cardinal $\tau\leq \lambda^{++}$, there is a forcing extension  where the following hold:
    \begin{enumerate}
        \item  $\kappa$ is supercompact;
        \item  $\lambda$ is measurable;
        \item $\lambda$ carries exactly $\tau$ normal measures.
    \end{enumerate}
\end{theorem}

\begin{proof}
    Let $\mathbb{L}$ be the Laver preparation rendering the supercompactness of $\kappa$ indestructible under $\kappa$-directed closed forcing (here one may use either Theorem~\ref{theorem: Laver indestructibility} or Laver’s original construction in \cite{Lav}, since we are not concerned with controlling the measures on $\kappa$. If, in addition, one wishes to preserve cardinals and GCH after forcing with $\mathbb{L}$—a feature that may be of independent interest—one may instead apply Theorem~\ref{theorem: Laver indestructibility+GCH}). Let $G\subseteq \mathbb{L}$ be generic over $V$. By the L\'evy–Solovay Theorem (see \cite{levysolovay1967}), $\lambda$ remains measurable in $V[G]$ and carries a unique normal measure $U$, necessarily of trivial Mitchell rank.

If $\tau=(\lambda^{++})^{V[G]}$, force with the Easton support iteration adding a single Cohen subset to every inaccessible cardinal in the interval $(\kappa,\lambda]$ (alternatively, apply Lemma~\ref{lemma: kappa^++ many normal measures}, replacing $(\kappa_0,\kappa)$ in the statement of the lemma with $(\kappa,\lambda)$). This blows up the number of normal measures on $\lambda$ to $(\lambda^{++})^{V[G]}$, while preserving the supercompactness of $\kappa$.

If instead $\tau\leq (\lambda^{+})^{V[G]}$, use the splitting forcing $\mathbb{P}^{\tau,I}$, where $I\subseteq \kappa$ is the set of inaccessible cardinals above $\lambda$. By Theorem~\ref{thm: properties of the splitting forcing}, if $H\subseteq \mathbb{P}^{\tau,I}$ is generic over $V[G]$, then every normal measure $W\in V[G][H]$ on $\lambda$ is of the form $U^*_\eta$ for some $\eta<\tau$.
\end{proof}

\subsection{The first measurable limit of supercompact cardinals}

In this section, we study the number of normal measures on the least cardinal that is a measurable limit of supercompact cardinals. Let us first explain why this cardinal is of interest.

By a well-known theorem of Menas, every measurable cardinal that is a limit of strongly compact cardinals is strongly compact (see \cite[Theorem 2.21]{menas1974strong}). In a recent landmark result, Goldberg proved that, under the Ultrapower Axiom, every strongly compact cardinal that is not supercompact is a measurable limit of supercompact cardinals (see \cite[Theorem 8.3.10]{goldbergtheultrapoweraxiom}). It follows that, under UA, Menas’ theorem provides essentially the only way to distinguish between strongly compact and supercompact cardinals.

Assuming the UA or the linearity of the Mitchell order, the least cardinal $\kappa$ that is a measurable limit of supercompact cardinals  
is a strongly compact cardinal which carries a unique normal measure (hence, it cannot be supercompact). Indeed, if $\kappa$ carried two distinct normal measures, then one of them would have nonzero Mitchell order. Consequently, the fact that $\kappa$ is measurable and a limit of supercompact cardinals would remain true in the corresponding ultrapower. In particular, this property would reflect to a smaller cardinal, contradicting the minimality of $\kappa$.

The next theorem shows how to control the number of normal measures on the least measurable limit of supercompact cardinals. 

Let us remark that the proof technique below can be applied as well to the least measurable limit of strong cardinals. By a Theorem of Hamkins (see \cite[Corollary 2.7]{HamkinsTall}), measurable cardinals which are limits of strong cardinals are tall cardinals. The number of normal measures on tall cardinals was extensively studied in \cite{ApterCummingsTall}. Theorem \ref{thm: first limit of supercompacts} provides alternative proofs of some of the results obtained there.

\begin{theorem}\label{thm: first limit of supercompacts} 
  Assume the $\mathrm{GCH}$ holds, and  suppose that $\kappa$ is the first measurable limit of supercompact (strong) cardinals. Assume that $\kappa$ carries a unique normal measure of Mitchell order 0. Let $\tau\leq \kappa^{++}$ be a cardinal. Then, there is  a generic extension where $\kappa$ remains the first measurable limit of supercompact (strong) cardinals, and $\kappa$ carries exactly $\tau$ normal measures. 
\end{theorem}
\begin{proof}
  Let $\kappa$ be the first measurable limit of supercompact (strong) cardinals. Denote by $I$ the set of all inaccessible cardinals below $\kappa$ that are limits of supercompact (strong) cardinals. Our proof proceeds by forcing with a two-step iteration $\mathbb{Q}=\Add(\omega,1)\ast\mathbb{P}$, where $\mathbb{P}$ is an $I$-spaced nonstationary support iterated forcing ($\po$ will be formally defined below), so that $\mathbb{Q}$ exhibits a gap at $\omega_1$. In particular, forcing with $\mathbb{Q}$ does not create new supercompact, strong or measurable cardinals, by Hamkins' Gap Forcing Theorem (Theorem \ref{thm: Gap forcing theorem}). Thus, if $\mathbb{P}$ preserves the fact that $\kappa$ is a measurable limit of supercompact (strong) cardinals, then it also preserves that $\kappa$ is the least such cardinal.

   Assume first that $\tau\leq \kappa^{+}$. Let $\mathbb{P}=\langle \mathbb{P}_\alpha,  \dot{\mathbb{Q}}_\alpha\colon \alpha<\kappa\rangle$ be the $I$-spaced nonstationary support iteration where for each $\alpha\in I$ the $\mathbb{P}_\alpha$-name  $\dot{\mathbb{Q}}_\alpha$ is taken as follows:
   \begin{enumerate}
       \item If $\alpha\in I$ is a limit point of $I$, $\dot{\qo}_\alpha$ is a $\mathbb{P}_\alpha$-name for the atomic forcing  $\{\one\}\cup f_\tau(\alpha)$ (here $f_\tau$ is the canonical function associated with $\tau$, and $\dot{\qo}_\alpha$ is as in the definition of the Splitting forcing in Section \ref{Section: splitting forcing}).
       \item If $\alpha\in I$ is not a limit point of $I$, let $\lambda_\alpha$ be the least supercompact (strong) cardinal strictly above $\alpha$, and let $\dot{\qo}_\alpha$ be the $\po_\alpha$-name for the Laver preparation $\mathbb{L}(\alpha^+, \lambda_\alpha)$ from Notation \ref{notation: Laver} (for the ``strong" version of the theorem, $\dot{\qo}_\alpha$ is taken to be $\mathbb{GS}(\alpha^+, \lambda_\alpha)$ from Notation \ref{notation: Laver}).
   \end{enumerate}

    Let $G\subseteq \qo$ be generic over $V$. 

    We first argue that $\kappa$ remains a limit of supercompact (strong) cardinals in $V[G]$. It suffices to prove that, for every cardinal $\alpha\in I$ which is not a limit point of $I$, $\lambda_\alpha$ remains supercompact (strong) cardinal in $V[G]$. Indeed, factor
    $$\qo = {\rm{Add}}(\omega,1) \ast \mathbb{P}_{\alpha} \ast \dot{\qo}_\alpha\ast  \mathbb{P}\setminus (\alpha+1)$$
    and note that ${\rm{Add}}(\omega,1) \ast \mathbb{P}_{\alpha} $ preserves the supercompactness (strongness) of $\lambda_\alpha$ as it is a small forcing relative to $\lambda_\alpha$; $\dot{\qo}_\alpha$ preserves the supercompactness (strongness) of $\lambda_\alpha$ and makes it indestructible under $(\lambda_\alpha)^+$-directed closed forcings; and $\po\setminus (\alpha+1)$ is $(\lambda_\alpha)^+$-directed closed since $\min(I\setminus (\alpha+1)) > \lambda_\alpha$. Hence $\lambda_\alpha$ remains a supercompact (strong) cardinal in $V[G]$.

    Next, we argue that $\kappa$ remains measurable and carries exactly $\tau$ normal measures in $V[G]$. By the remarks preceding the theorem, $\kappa$ carries a unique normal measure $U$ in $V$. We now aim to invoke Lemma~\ref{Lemma: lifting elementary embeddings with fusion} to lift the embedding $j_U\colon V\to M_U$ to $V[G]$. Before doing so, we verify that the assumptions of the lemma are satisfied. For Clause~$(\beth)$, observe that $\kappa$ is the sole generator of $j_U$ and is therefore bounded by $j_U(\alpha\mapsto\alpha^+)(\kappa)=\kappa^+$, which in turn is smaller than $\min(j_U(I)\setminus (\kappa+1))$.
    
    Since $\kappa\in j_U(I)$, the same argument as in Lemma~\ref{Lemma: How embeddings lift after the splitting forcing} shows that $j_U\colon V\to M_U$ lifts to $V[G]$ in exactly $\tau$ ways. More precisely, the embeddings $j^*\colon V[G]\to M^*$ extending $j_U$ are of the form
    $$j^*_\eta\colon V[G]\rightarrow M_U[G\ast \{\langle \kappa,\eta\rangle\}\ast (j_{U}[G]\setminus \kappa+1)]$$
    for some $\eta<\tau$, and each $j^*_\eta$ induces a distinct normal measure $U^*_\eta$ on $\kappa$ in $V[G]$. 

    On the other hand, suppose that $W\in V[G]$ is an arbitrary normal measure in $\kappa$. By the Gap Forcing Theorem, $W$ lifts a normal measure $U\in V$. Since $\kappa$ is a measurable limit of supercompact (strong) cardinals in $V$, $I\in U$. In particular $I\in W$ and thus $\kappa\in j_W(I)$. Arguing exactly as in the last paragraph of the proof of Theorem~\ref{thm: properties of the splitting forcing} we conclude that $W=U^*_\eta$ for some $\eta<\tau$. Thereby, $\kappa$ carries exactly $\tau$ normal measures in $V[G]$.

    This concludes the proof of the theorem for $\tau\leq \kappa^+$. If $\tau = \kappa^{++}$, modify $\po$ to an iterated forcing of length $\kappa+1$, such that, for every $\alpha\in I\cup \{\kappa\}$ which is a limit point of $I$, $\dot{\qo}_\alpha$ is forced to be $({\rm{Add}}(\alpha,1))^{V^{\po_\alpha}}$. Standard arguments show that $\kappa$ remains measurable and carries $\kappa^{++}$ normal measures in $V[G]$. The same argument as above shows that $\kappa$ is a limit of supercompact (strong) cardinals in $V[G]$.
\end{proof}

\begin{remark}\label{Remark: adding GCH to the theorem about measurable limit of supercompacts}
    In Theorem \ref{thm: first limit of supercompacts} we can obtain GCH in the generic extension by using the version of the indestructibility theorem under forcings which preserve GCH (\ref{theorem: Laver indestructibility+GCH}).
\end{remark}

\subsection{A generalization of the Goldberg--Woodin Theorem}

One of the main results in \cite{GitikKaplan} is the proof that, assuming the consistency of the Ultrapower Axiom\footnote{As mentioned above, weaker consequences of UA suffice for the proof.} with a supercompact cardinal, it is consistent that the least measurable cardinal is strongly compact and has a unique normal measure. The proof was performed over a model of the Ultrapower Axiom in which there is a supercompact cardinal, and there are no measurable cardinals above it (this configuration could be achieved by cutting the universe at the least measurable cardinal above the least supercompact cardinal). The first author asked whether the anti-large cardinal component of the proof can be omitted. In an unpublished work prior to \cite{GitikKaplan}, Goldberg and Woodin proved  the same result without relying on any anti-large cardinal assumption, starting from a model of UA with a  measurable limit of supercompact cardinals. Building on the Goldberg--Woodin approach, we prove that under the same assumption, the first measurable cardinal can be strongly compact and have any prescribed number of normal measures, without appealing to anti-large cardinal assumptions in the proof.

\begin{theorem}\label{thm: Goldberg-Woodin}
    Assume the {\rm{GCH}} holds, $\kappa$ is a  measurable limit of supercompact cardinals, $\tau\leq \kappa^{++}$, and $\kappa$ has a unique normal measure of Mitchell order 0.\footnote{If $\kappa$ is the least measurable limit of supercompact cardinals and the Mitchell order is linear, then, by its minimality, $\kappa$ carries a unique normal measure of order 0. Thus, under UA, the least measurable limit of supercompact cardinals fulfills the assumptions in the theorem.} Then there is a generic extension where $\kappa$ is the least measurable cardinal, the least strongly compact cardinal, and $\kappa$ carries exactly $\tau$ normal measures.
\end{theorem}

\begin{proof}
    By Theorem \ref{thm: first limit of supercompacts} and Remark \ref{Remark: adding GCH to the theorem about measurable limit of supercompacts}, we can assume that GCH holds, $\kappa$ is the least measurable limit of supercomapct cardinals, and $\kappa$ has $\tau$ normal measures. Next, perform a Magidor iteration $\po = \la \po_\alpha, \dot{\qo}_\alpha \colon \alpha<\kappa \ra$ that adds a Prikry sequence to any measurable cardinal below $\kappa$ (including the supercomapct cardinals below $\kappa$). In the resulting model, $\kappa$ is the least measurable cardinal and a strongly compact cardinal (see \cite{MagidorIdentity} or \cite{Gitik-handbook} for more details on the Magidor iteration of Prikry forcings and the preservation of the strongly compact cardinal). Finally, by \cite[Theorem 0.1]{KaplanMagidorIterations}, $\kappa$ has exactly $\tau$ normal measures in the generic extension.
\end{proof}

\section{A comment on Woodin's HOD hypothesis}

Woodin’s $\HOD$ hypothesis asserts that there is a proper class of regular cardinals that are not $\omega$-strongly measurable in $\HOD$. It is motivated by Woodin’s landmark $\HOD$ Dichotomy Theorem:

\begin{theorem}[Woodin \cite{midrasha}]
Let $\kappa$ be an extendible cardinal. Then exactly one of the following holds:
\begin{enumerate}
\item $\HOD$ is a weak extender model for the supercompactness of $\kappa$.
\item Every regular cardinal $\geq \kappa$ is $\omega$-strongly measurable in $\HOD$.
\end{enumerate}
\end{theorem}

We refer to \cite{midrasha} for the relevant definitions. 

In \cite{GoldbergPoveda}, Goldberg-Poveda proved the consistency of Woodin's $\HOD$ hypothesis with the first supercompact cardinal $\kappa$ being strongly compact in $\HOD$ yet not $2^\kappa$-supercompact in $\HOD$. Therefore, if $\kappa$ is merely supercompact (and not $\HOD$-supercompact, as required by Woodin),  $\HOD$ may fail to be a weak extender model for the supercompactness of the first supercompact cardinal $\kappa$ -- even under the HOD hypothesis.

In this brief section we further strengthen this disagreement between $V$ and $\HOD$ by producing the consistency of the $\HOD$ hypothesis with the first supercompact cardinal  carrying a unique normal measure in $\HOD$.

\smallskip

We will need the natural  nonstationary support variation of the family of self-coding iterations considered in \cite{GoldbergPoveda}:

\begin{definition}\label{def: selfcoding}    Let $\theta$ be an ordinal. A nonstationary support \emph{self-coding iteration} is a sequence $$\mathbb{P} = \langle \mathbb{P}_\alpha, \dot{\mathbb{Q}}_\alpha,
    \dot{A}_\alpha\colon \alpha<\theta\rangle$$ such that the following hold {for every Mahlo cardinal $\alpha$}, where $t_\alpha$ denotes the transitive closure of $\mathbb{P}_\alpha$ and $G_\alpha$ denotes the canonical $\mathbb P_\alpha$-name for a generic set:
    \begin{enumerate}
        \item $\dot{\mathbb{Q}}_\alpha$ is a $\mathbb{P}_\alpha$-name for a poset that is forced by the trivial condition  to be $|\mathbb{P}_\beta|$-closed, for all $\beta<\alpha.$
        \item $\dot{A}_\alpha$ is a $\mathbb{P}_\alpha$-name forced by
        the trivial condition to be a binary relation on $|t_\alpha|^V$ whose transitive collapse is $t_\alpha\cup\{\dot{G}_\alpha\}.$
        \item $\mathbb{P}_{\alpha+1}=\mathbb{P}_\alpha\ast \dot{\mathbb{Q}}_\alpha\ast {\mathrm{Code}}(\dot{A}_\alpha,\gamma)$ where $\gamma$ is the first inaccessible above $|\mathbb{P}_\alpha\ast \dot{\mathbb{Q}}_\alpha|$ and $\mathrm{Code}(\dot{A}_\alpha,\gamma)$ is the coding poset from \cite[\S2]{GoldbergPoveda}.
        
    \end{enumerate}
    As usual, when defining $\po$ at limit stages $\alpha$, the iteration $\mathbb{P}_\alpha$ is constructed as the nonstationary support limit if $\alpha$ is inaccessible 
        or as the inverse limit otherwise.
\end{definition}
The following fact was proved in {\cite[Lemma 2.21]{GoldbergPoveda}} for self-coding iterations with Easton support. The proof for self-coding iterations with nonstationary support is similar: 
\begin{fact}\label{lemma: self-coding}\label{fact: V=HOD}
    Suppose that $\theta$ is a Mahlo cardinal and let $\vec{\mathbb{P}}=\langle \mathbb{P}_\alpha, \dot{A}_\alpha, \dot{\mathbb{Q}}_\alpha\colon \alpha<\theta\rangle$ be a a self-coding iteration. If $G\s \mathbb{P}_\theta$ is generic over $V$ then $G$ is definable over the structure $(V[G]_\theta, V_\theta, \vec{\mathbb{P}})$ without parameters.  In particular, if $V\s \HOD^{V[G]}$ and $\vec{\mathbb{P}}$ is definable over $V_\theta$ then $$\HOD^{V[G]}=V[G].$$
\end{fact}

\begin{theorem}\label{thm: Hod hypothesis}
Assume $\mathrm{GCH}$ and $V=\gHOD$ both hold\footnote{Recall that ${\rm{gHOD}} := \{ x\in V \colon \text{ for every poset }\po\in V, \ (\Vdash_{\po}  \check{x}\in {\rm{HOD}}) \}$. See \cite{Geology}.}, $\kappa$ is a supercompact cardinal, {there are no inaccessible cardinals above $\kappa$} and there exists a unique normal measure on $\kappa$ with Mitchell order $0$. Then, the following configuration is consistent:
    \begin{enumerate}
        \item The $\mathrm{HOD}$ hypothesis holds.
        \item $\kappa$ is supercompact.
        \item In $\HOD$, $\kappa$ is strongly compact  and it carries exactly one normal measure.
    \end{enumerate}    
\end{theorem}
\begin{proof}
Fix $\ell\colon \kappa\rightarrow V_\kappa$ a Laver function. We shall denote by $I$ the set of measurable cardinals $\alpha<\kappa$ that are closure points of $\ell$, namely
$$I = \{ \alpha<\kappa \colon \alpha \text{ is measurable }\land  \forall \beta<\alpha \left( \ell(\beta)\in {\rm{ORD}}\to \ell(\beta)<\alpha \right) \}.$$
Let $\mathbb{P}_\kappa$ the  $I$-spaced, nonstationary support self-coding iteration of length $\kappa$ such that $\dot{\mathbb{Q}}_\alpha$ is forced to be $\Add(\alpha,1)$ at every $\alpha\in I$.

\smallskip

By Fact~\ref{fact: V=HOD},  for every $G\s \mathbb{P}_\kappa$ generic over $V$ we have that $$V[G]=\HOD^{V[G]}.$$

 Our intended generic extension is obtained by forcing with $\mathbb{P}_\kappa\ast \dot{\Add}(\kappa,1)$. Let $G\ast H\s \mathbb{P}_\kappa\ast \Add(\kappa,1)$ a generic over $V$.
  $$\HOD^{V[G][H]}=V[G]$$ by homogeneity of $\Add(\kappa,1)^{V[G]}$ and because this forcing does not disturb the coding into the power-set function pattern (see  \cite{GoldbergPoveda} for details)

\begin{claim}\label{claim: a unique normal measure}
    In  $\HOD^{V[G][H]}$, $\kappa$ carries exactly one normal measure.
\end{claim}
\begin{proof}
By our previous observations, it suffices to argue the claim in  $V[G]$. Let $U$ be the unique normal measure of trivial Mitchell rank in $V$. By the argument of Claim~\ref{Claim: indestructibility proof, claim 1}, $U$ generates in $V[G]$ a normal measure $U^*$.   

Suppose for the sake of a contradiction $W$ is a normal measure on $\kappa$ in $V[G]$ different from $U^*$. Then, by the Gap Forcing Theorem,  $\bar{W}=W\cap V$ belongs to $V$ and the ultrapower of $V[G]$ under $W$, $M_W$, is of the form $M_W=M[j_W(G)]$ where  $M=M_W\cap V$. 

{Crucially, $M=M_{\bar W}$.} To show this, let $k\colon M_{\bar{W}}\to M$ be the map defined by $$k([f]_{\bar{W}}) := [f]_{W}.$$ It is not hard to verify that $k$ is elementary and that $j_W\restriction V = k\circ j_{\bar{W}}$. To complete the proof of the equality it is  enough to argue that $k$ is the identity function. Assume otherwise, and denote $\mu = \crit(k)$. Since $\bar{W}, W$ are normal measures, $\mu\geq (\kappa^{++})^{M_{\bar{W}}}$. Write $\mu = [f]_W$ for some $f\colon \kappa\to \ord$ in $V[G]$. By Lemma \ref{lemma: capturing functions via fusion}, there exists $F\in V$ and a club $C\subseteq \kappa$ in $V$, such that, for every $\alpha\in C$, $f(\alpha)\in F(\alpha)$ and $|F(\alpha)|\leq|\mathbb{P}_{\alpha+1}|^V= \alpha^+$. In particular, by normality of $W$, $\kappa\in j_W(C)$, so
    $$\mu = j_W(f)(\kappa)  \in j_{W}(F)(\kappa) = k(j_{\bar{W}}(F))(\kappa) = k(j_{\bar{W}}(F)(\kappa)) = k[ j_{\bar{W}}(F)(\kappa)]$$
    where in the second-to-last equality we used the fact that  $\crit(k)>\kappa$, and in the last equality we used the fact that $|j_{\bar{W}}(F)(\kappa)|\leq \kappa^{+}< \crit(k)$. Overall, we proved that $\mu\in \text{Im}(k)$, contradicting the fact that $\mu = \crit(k)$.   

    \smallskip

Note that since we are assuming that $W\neq U^*$ it cannot be the case that $\bar{W}=U$. Therefore,  by our assumption in $V$, the measure $\bar{W}$ must have Mitchell-order ${\geq 1}$ in $V$. Thus we have an elementary embedding $j_W\colon V[G]\rightarrow M_{\bar W}[j_{{W}}(G)]$ 
where $M_{\bar W}\models ``\kappa$ is measurable". In particular, this means that $j_W(\mathbb{P}_\kappa)$ has added a Cohen subset to $M_{\bar W}[j_W(G)\cap j_W(\mathbb{P}_\kappa)_\kappa]=M_{\bar{W}}[G]$ at stage  $\kappa$. Next, we argue that this Cohen subset of $\kappa$ must be $V[G]$-generic which yields the sought contradiction with $W\neq U^*$.

First, $\Add(\kappa,1)^{V[G]}=\Add(\kappa,1)^{M_{\bar W}[G]}$ because $M_{\bar W}[G]$ is an inner model of $V[G]$ closed under $\kappa$-sequences of its elements by Claim \ref{Claim: closure under kappa sequneces and fusion}. Second, any maximal antichain $A\in V[G]$ in $({\rm{Add}}(\kappa, 1)^{V[G]}$ is a subset of $V[G]_\kappa$, so again by Claim \ref{Claim: closure under kappa sequneces and fusion}, $A\in M_{\bar{W}}[G]$.\footnote{Indeed, $A$ can be coded as a $\kappa$-sequence $\vec{X} = \la A\cap (V_\beta)^{V[G]} \colon \beta<\kappa \ra$, and, since $A\cap (V_\beta)^{V[G]}\in M_{\bar{W}}[G]$ for every $\beta<\kappa$, $\vec{X}\in M_{\bar{W}}[G]$.}
Thus, any $M_{\bar{W}}[G]$-generic for $\Add(\kappa,1)^{V[G]}$ is generic over $V[G]$.
\end{proof}

An argument simpler than the one given in the proof of Theorem \ref{thm: Strong compactness after iterating nonreflecting stat sets} shows that $\kappa$ remains strongly compact in $V[G]$.\footnote{The simplicity comes from the fact that this time our iteration uses Cohen forcing in place of $\mathbb{NR}(\alpha)$; that is, now the forcing at each stage $\alpha$ is $\alpha$-directed closed, while before, the forcings $\mathbb{NR}(\alpha)$ only had arbitrarily high dense directed closed subsets.}

\begin{claim}
In $V[G]$, $\kappa$ is strongly compact.    
\end{claim}
\begin{proof}
   Let $\lambda>\kappa$ be regular,  $U_0$  the unique normal measure on $\kappa$ with trivial Mitchell rank and $F\in M_{U_0}$ a normal, fine ultrafilter on $\mathcal{P}_{j_{U_0}(\kappa)}(j_{U_0}(\lambda))^{M_U}$ such that $j_{F}(j_{U_0}(\ell))(j_{U_0}(\kappa))=j_{U_0}(\lambda).$
   We will lift the  embedding $$j=j_{F}\circ j_{U_0}\colon V\to M\simeq \Ult(M_{U_0}, F).$$
   
   First, since $\mathbb{P}_\kappa$ has nonstationary support and $U_0$ has a trivial Mitchell order, $j_{U_0}$ lifts to 
   $$j_{U_0}\colon V[G]\rightarrow M_{U_0}[G\ast j_{U_0}[G]\setminus \kappa+1]=M_{U_0}[j_{U_0}[G]].$$
   Second, by elementarity and the definition of $\mathbb{P}_\kappa$, $j(\mathbb{P}_\kappa)$ factorises as $$j_{U_0}(\mathbb{P}_\kappa)\ast \dot{\Add}(j_{U_0}(\kappa), 1)\ast \dot{\mathbb{Q}},$$
   where the latter poset is forced to be $j_{U_0}(\lambda)^{+M_U[j_{U_0}(G)]}$-directed closed.

   Let $g\in V[G]$ be an $M[j_{U_0}(G)]$-generic filter for $\Add(j_{U_0}(\kappa), 1)^{M[j_{U_0}[G]]}$. Constructing such $g$ in $V[G]$ is possible, since $({\rm{Add}}(j_{U_0}(\kappa),1))^{M[ j_{U_0}[G] ]}$ is a $\kappa^+$-closed forcing whose dense open sets in $M[j_{U_0}[G]]$ can be enumerated in $V[G]$ in order-type $\kappa^+$.

  Since $\dot{\qo}$ is sufficiently directed closed, the set $$\{j(p)\setminus j_{U_0}(\kappa)\colon p\in G\}=\{\{\one\}^\smallfrown j(p)\setminus j_{U_0}(\kappa)+1\colon p\in G\}$$
  admits a lower bound $q\in j(\po)\setminus j_{U_0}(\kappa)$ with $q(j_{U_0}(\kappa)) = \one$. Using $q$ as a master condition, we can build $L\in M_U[j_{U_0}(G)\ast g]$ a $j(\po)\setminus (j_{U_0}(\kappa)+1)$-generic over $M[j_{U_0}[G]\ast g]$ such that $q\in g*L$. Therefore, $j\colon V\to M$ lifts to
  $$j\colon V[G]\rightarrow M[j_{U_0}(G)\ast g\ast L]$$
  that is a $\lambda$-strongly compact embedding.
\end{proof}

\begin{claim}
    In $V[G]$, $\kappa$ is not $2^\kappa$-supercompact.
\end{claim}
\begin{proof}
    Suppose otherwise and let $j\colon V[G]\rightarrow M$ be a witnessing embedding in $V[G]$. Then $\kappa$ is measurable in $M$. By the Gap Forcing Theorem \ref{thm: Gap forcing theorem}, $M=\bar{M}[H]$ for some subclass $\bar{M}\s V$ and a $j(\mathbb{P}_\kappa)$-generic $H$ over $\bar{M}$. 

  By the Gap Forcing Theorem~\ref{thm: Gap forcing theorem} forcing with $j(\mathbb{P}_\kappa)$ does not create new measurable cardinals. As a result, $\kappa$ is measurable in $\bar{M}$ as well. In addition, $\kappa$ is a closure point of $j(\ell)$  so, all in all, $\kappa\in j(I)$. This means that at stage $\kappa$ the iteration $j(\mathbb{P}_\kappa)$ has opted to force with $\Add(\kappa,1)$. Using exactly the same argument given in Claim \ref{claim: a unique normal measure} one demonstrates that this Cohen is in fact generic over $V[G]$, which produces the sought contradiction.
   
 \end{proof}
After forcing with  $H\s \mathrm{Add}(\kappa,1)_{V[G]}$ the supercompactness of $\kappa$ is restored: Indeed,  in $M[G]$ the forcing used at stage $\kappa$ is of the form $\Add(\kappa,1)_{V[G]}\ast \mathrm{Code}(A_\kappa,\gamma)_{M[G]}$ where $\gamma$ is the first $M[G]$-inaccessible above $\kappa$. Since we were assuming that no inaccessible cardinal greater than $\kappa$ exists it follows that the first $M[G]$-inaccessible greater than $\kappa$ is above $\lambda$. As a result, standard arguments allow us to define a $M[G]$-generic filter for $\mathrm{Code}(A_\kappa,\gamma)_{M[G]}$ in $V[G\ast H]$. These considerations together with standard arguments of Silver (see \cite[Theorem~12.6]{CummingsHandbook}) finally show that $\kappa$ is supercompact in $V[G\ast H]$.
\end{proof}

\bibliographystyle{plain} 
\bibliography{citations}
\end{document}